\documentclass[journal]{IEEEtran}

\usepackage{amsmath,amscd,amsbsy,amssymb,latexsym,url,bm,amsthm}
\usepackage{array}
\usepackage{algorithm}
\usepackage{algorithmic}
\usepackage{booktabs}
\usepackage{cite}
\usepackage{etoolbox} 
\preto\subequations{\ifhmode\unskip\fi} 
\usepackage[T1]{fontenc}
\usepackage{graphicx}
\RequirePackage{doi}
\usepackage{hyperref}
\hypersetup{hidelinks}
\usepackage{multirow,makecell}
\usepackage{subeqnarray}
\usepackage[caption=false,font=footnotesize,subrefformat=parens,labelformat=parens]{subfig}
\usepackage{tabularx}
\usepackage{threeparttable}
\usepackage{tikz}
\usepackage{xcolor}
\usepackage{balance}
\allowdisplaybreaks

\newtheorem{assumption}{Assumption}
\newtheorem{theorem}{Theorem}
\newtheorem{lemma}[theorem]{Lemma}

\newtheorem{corollary}{Corollary}

\newtheorem{remark}{Remark}

\newcommand{\bigO}[1]{\ensuremath{\mathop{}\mathopen{}\mathcal{O}\mathopen{}\left(#1\right)}}
\renewcommand{\bar}[1]{\mkern 1.5mu\overline{\mkern-1.5mu#1\mkern-1.5mu}\mkern 1.5mu}
\newcommand{\ud}{\mathrm{d}}
\newcommand{\RNum}[1]{\uppercase\expandafter{\romannumeral #1\relax}}
\newcommand{\numcircled}[1]{\tikz[baseline=(char.base)]{
            \node[shape=circle,draw,inner sep=0.7pt] (char) {#1};}}

\newcommand{\obj}{\tilde{\Phi}}
\DeclareMathOperator*{\argmin}{arg\,min}
\DeclareMathOperator{\E}{\mathbb{E}}

\DeclareMathOperator{\prox}{prox}
\DeclareMathOperator{\sat}{sat}

\newcolumntype{Y}{>{\centering\arraybackslash}X}

\makeatletter
\def\raisedotfill{%
  \leavevmode
  \cleaders \hb@xt@ .44em{\hss\raise0.5ex\hbox{.}\hss}\hfill
  \kern\z@}
\pretocmd\@bibitem{\color{black}\csname keycolor#1\endcsname}{}{\fail}
\newcommand\citecolor[1]{\@namedef{keycolor#1}{\color{blue}}}
\makeatother

\begin{document}
    \title{Model-Free Nonlinear Feedback Optimization}
    \author{Zhiyu He\textsuperscript{$\dagger$,$\ddagger$},~\IEEEmembership{Student Member, IEEE}, Saverio Bolognani\textsuperscript{$\dagger$},~\IEEEmembership{Member, IEEE}, Jianping He\textsuperscript{$\ddagger$},~\IEEEmembership{Senior Member, IEEE}, \\ Florian D{\"o}rfler\textsuperscript{$\dagger$},~\IEEEmembership{Senior Member, IEEE}, and Xinping Guan\textsuperscript{$\ddagger$}~\IEEEmembership{Fellow, IEEE}
        \thanks{This work was supported in part by the National Natural Science Foundation of China under Grants 62373247, 92167205, 61933009, and 62025305, in part by the Max Planck ETH Center for Learning Systems, and in part by the Swiss National Science Foundation through the NCCR Automation under Grant 180545. \emph{(Corresponding author: Xinping Guan.)}

        \textsuperscript{$\dagger$}Automatic Control Laboratory, ETH Z{\"u}rich, 8092 Z{\"u}rich, Switzerland (email: \{zhiyhe, bsaverio, dorfler\}@ethz.ch).


        \textsuperscript{$\ddagger$}Department of Automation, Shanghai Jiao Tong University, Shanghai 200240, China, Key Laboratory of System Control and Information Processing, Ministry of Education of China, Shanghai 200240, China, and Shanghai Engineering Research Center of Intelligent Control and Management, Shanghai 200240, China (email: \{hzy970920, jphe, xpguan\}@sjtu.edu.cn).


        }}

    \maketitle

    \begin{abstract}
        Feedback optimization is a control paradigm that enables physical systems to autonomously reach efficient operating points. Its central idea is to interconnect optimization iterations in closed-loop with the physical plant. Since iterative gradient-based methods are extensively used to achieve optimality, 
        feedback optimization controllers typically require the knowledge of the steady-state sensitivity of the plant, which may not be easily accessible in some applications. In contrast, in this paper, we develop a model-free feedback controller for efficient steady-state operation of general dynamical systems. The proposed design consists of updating control inputs via gradient estimates constructed from evaluations of the nonconvex objective at the current input and at the measured output. We study the dynamic interconnection of the proposed iterative controller with a stable nonlinear discrete-time plant. For this setup, we characterize the optimality and stability of the closed-loop behavior as functions of the problem dimension, the number of iterations, and the rate of convergence of the physical plant. To handle general constraints that affect multiple inputs, we enhance the controller with Frank-Wolfe-type updates.
    \end{abstract}


\section{Introduction}\label{sec:introduction}
Efficient operation is a fundamental design goal for engineering systems, including (but not limited to) communication networks, power grids, transportation systems, and process control systems. Although numerical optimization methods \cite{nesterov2018lectures} have been extensively studied as a way to obtain optimal decisions based on the exact abstraction of the problem, achieving the same on a real plant remains a challenging task. The main inhibiting factors include the absence of precise knowledge of plant models and the existence of unmeasured disturbances. For these reasons, the approach of first running an optimization program offline and then commanding its solution to a plant in a feedforward manner is rarely effective in practice. In contrast, feedback control based on real-time measurements is robust, albeit rarely optimal and able to handle constraints. Hence, many approaches have been developed to combine the complementary benefits of feedback control and feedforward optimization.

\subsection{Related Work}
Extremum seeking (ES)\cite{ariyur2003real} is one of the first approaches that realized the above combination without relying on the plant model (\emph{model-free}), and the first rigorous stability proof of ES feedback is in \cite{krstic2000stability}. ES uses perturbations (e.g., sinusoidal signals) for exploration, collects measurements, estimates gradients via averaging, and then updates inputs. Traditional ES is mostly suitable for low-dimensional systems, where the orthogonality requirement on the elements of perturbation signals is easy to satisfy\cite{manzie2009extremum}. To handle high dimensionality, stochastic ES with perturbations governed by stochastic processes\cite{liu2016stochastic} and Newton-based ES\cite{ghaffari2012multivariable,liu2014newton} are developed. Some work on stochastic ES adopts the perspective of stochastic approximation\cite{spall1992multivariate} and examines quadratic maps without dynamics\cite{manzie2009extremum}. Advanced multivariate ES tackles general nonlinear dynamics that appear as partial differential equations and model input/output delays\cite{oliveira2016extremum} or actuation dynamics\cite{oliveira2020multivariable}. Moreover, the common practice for ES to address constraints is to encode them in the objective via penalty functions\cite{guay2015constrained}. This practice may not always guarantee precise constraint enforcement. An exception is \cite{chen2021model}, which incorporates projected primal-dual gradient dynamics to satisfy constraints. In the above works, the convergence is established in the asymptotic sense, but how it depends on the problem dimension and the stability properties of the plant is unclear.


Reinforcement learning (RL) is another category of model-free approaches, 
through which agents interact with an uncertain environment, learn from the feedback, and make sequential decisions to maximize their cumulative reward\cite{sutton2018reinforcement}. This paper is, nonetheless, concerned with optimizing the steady-state operation of general stable dynamical systems, rather than the cumulative performance in the infinite horizon with the environment described by a Markov decision process\cite{sutton2018reinforcement}.

As an emerging paradigm, feedback optimization (FO) \cite{hauswirth2021optimization,simonetto2020time} offers a promising approach to drive general systems to efficient operating points, which constitute the optimal solutions of problems involving steady-state inputs and outputs. The key idea of FO is to implement optimization algorithms as feedback controllers, which are connected with physical plants to form a closed loop. Various approaches also share this closed-loop optimization viewpoint, including modifier adaptation\cite{marchetti2009modifier}, real-time iteration schemes\cite{diehl2005real}, and real-time model predictive control\cite{zeilinger2011real}. Compared to these approaches, FO requires limited model information and less computational effort (see the review in \cite{hauswirth2021optimization}). By utilizing real-time measurements of system outputs, FO circumvents the need to access input-output models to evaluate the gradient of the objective at the current operating point. Hence, the control inputs can be effectively updated also for problems with complex dynamics, e.g., real-time optimal power flow in electrical networks\cite{molzahn2017survey} and congestion control in communication networks\cite{low2002internet}. Furthermore, the feedback structure contributes to robustness against unmeasured disturbances and model mismatch\cite{ortmann2020experimental}, e.g., polytopic uncertainty in the plant Jacobian\cite{colombino2019towards,blanchini2016model}, as well as autonomous tracking of trajectories of optimal solutions of time-varying problems\cite{bernstein2019online,colombino2020online,tang2018feedback,chen2020model,nonhoff2020online,dall2020optimization}. 
Some works\cite{bernstein2019online,tang2018feedback,chen2020model,haberle2020non} consider fast-stable plants that are abstracted as algebraic steady-state maps. Others take system dynamics into account and characterize sufficient conditions for the closed-loop stability, including in continuous-time\cite{hauswirth2020timescale,colombino2020online,simpson2021analysis,bianchin2021online} and sampled-data settings\cite{belgioioso2021sampled}. Specifically, among works that handle nonconvex objectives and nonlinear systems, \cite{tang2018feedback,haberle2020non} address discrete-time systems represented by algebraic maps, while \cite{hauswirth2020timescale} tackles continuous-time systems. The results on nonconvex feedback optimization for discrete-time nonlinear dynamical systems are still lacking.

\subsection{Motivations}
FO requires limited model information, i.e., the sensitivity of the steady-state input-output map of the physical plant. Such information is needed because FO controllers update inputs via first-order methods (e.g., gradient\cite{hauswirth2020timescale}, projected gradient\cite{haberle2020non}, and primal-dual saddle-point methods\cite{bernstein2019online,colombino2020online}). Given objective functions depending on steady-state inputs and outputs, due to the chain rule, their gradients with respect to inputs naturally contain the sensitivity terms. Recent works demonstrate that the sensitivity can be estimated based on system identification\cite{ortmann2020experimental}, recursive least-squares estimation\cite{picallo2021adaptive}, or data-driven methods that use past input-output data of open-loop linear systems\cite{nonhoff2021data,bianchin2021online}. Nevertheless, on the one hand, the estimation can be a highly nontrivial task accompanied by errors, thus causing the approximate optimality of FO\cite{colombino2019towards}. On the other hand, the uncertainties in and the complexity of engineering systems (e.g., volatile renewable energy sources in large-scale power systems\cite{molzahn2017survey}) may render the sensitivity of the plant costly to compute or even impossible to formulate, let alone feasible to estimate.

A different perspective offers us new insights to design an entirely model-free optimization-based feedback controller as in ES or RL\@. For current works on FO, it is the direct use of gradients in updates that leads to the requirement of sensitivity. In other words, the key to model-free implementations lies in optimization without gradients. In fact, the key technique behind the elegant ES approach is an indirect way of learning gradients via perturbations and averaging, though restrictions exist when confronted with high-dimensional systems and constrained problems, and the core of policy gradient approaches to RL is expressing elusive performance gradients as simple expectations\cite{silver2014deterministic}. In contrast, the so-called \emph{zeroth-order optimization}\cite{liu2020primer} constructs gradient estimates\cite{nesterov2017random,duchi2015optimal,shamir2017optimal,flaxman2005online,zhang2020improving} based on function evaluations and uses these estimates to form descent directions. Zeroth-order methods own comparable algorithmic structures and convergence rates with their first-order counterparts\cite{liu2020primer}. They are therefore well suited to serve as a building block for model-free FO\@. 

\subsection{Contributions}
Motivated by the above observations, we design a model-free feedback controller for efficient steady-state operation of discrete-time nonlinear dynamical systems. We inherit from ES and FO the ideas of using measurements of outputs and implementing optimization-based controllers in closed loop. More importantly, to facilitate complete model-free operation, we update control inputs based on gradient estimates constructed from current and previous evaluations of objective functions. In contrast to numerical optimization where these update rules originate, the dynamic nature of the plant brings new challenges, e.g., the access of real-time instead of steady-state measurements and the coupling in plant dynamics and controller adjustments. To address these challenges, we establish recursive inequalities of the expected values of the Lyapunov function of the plant and the second moments of gradient estimates along the descent steps of the controller. Then, with the closed-loop interconnection, we synthesize the coupled evolution of the above quantities to explicitly bound the convergence measures. Different from ES, we handle input constraints via Frank-Wolfe-type methods and provide a non-asymptotic characterization of performance with respect to the problem dimension and the rate of convergence of the dynamic plant. Compared to recent works that employ multi-point gradient estimates to fulfill model-free extremization\cite{poveda2017robust} and cooperative optimization\cite{tang2023zeroth}, we examine nonlinear system dynamics and only use a single actuation step per iteration. In Table~\ref{table:comparison}, we compare various works on FO addressing the unknown sensitivity of the plant.

The main contributions are summarized as follows.
\begin{itemize}
    \item To the best of our knowledge, this is the first work on feedback optimization while considering nonconvex objectives and nonlinear discrete-time dynamical systems in both unconstrained and constrained settings.
    \item We propose a novel model-free feedback controller that drives the system to the solution of a given optimization problem by updating control inputs via gradient estimates. These estimates are obtained from present and past evaluations of the objective at inputs and measured outputs and do not depend on any prior model information.
    \item We explicitly characterize the optimality of closed-loop solutions. It is measured by the second moment of the gradient of the objective function in the unconstrained scenario, and by the expected Frank-Wolfe dual gap in the setting with input constraint sets. We demonstrate that the optimality scales polynomially with the problem dimension and the reciprocal of the number of iterations, and we quantify its dependence on the stability properties of the physical plant.
\end{itemize}


\begin{table}[!tb]
\centering
\caption{Comparison of Works on FO with Unknown Sensitivity}
\label{table:comparison}
\begin{threeparttable}
    \begin{tabularx}{\linewidth}{c *{4}{Y}}
        \toprule
        \textbf{Works} & \cite{bianchin2021online,nonhoff2021data} & \cite{picallo2021adaptive} & \cite{chen2020model} & this work \\
        \midrule
        \textbf{\makecell{Nonconvex\\ Objectives}} & & & & \checkmark \\
        \midrule 
        \textbf{\makecell{Input\\ Constraints}} & \checkmark & \checkmark & \checkmark & \checkmark \\
        \midrule
        \textbf{\makecell{Plant\\ Model}} & \makecell{linear\\ dynamic} & \makecell{nonlinear\\ algebraic} & \makecell{linear\\ algebraic} & \makecell{nonlinear\\ dynamic} \\
        \midrule
        \textbf{\makecell{Method}} & 
        {\scriptsize \makecell{data-driven\\ sensitivity \\ estimation \\ based on \\ input-output \\ trajectories}} & 
        {\scriptsize \makecell{recursive \\ least-squares \\ sensitivity \\ estimation}} &
        \makecell{two-point \\ symmetric \\ gradient\\ estimation} & \makecell{residual \\ one-point \\ gradient\\ estimation} \\
        \bottomrule
        \addlinespace[0.5ex]
    \end{tabularx}
\end{threeparttable}
\end{table}

\subsection{Organization}
The rest of this paper is organized as follows. Section~\ref{sec:formulation} formally defines the problem of interests and provides some preliminaries. In Section~\ref{sec:design}, we present the proposed model-free nonconvex feedback optimization controller. Section~\ref{sec:analysis} provides the performance analysis of the interconnection of the controller and the dynamical plant. The extension to handle input constraint sets is explored in Section~\ref{sec:extension}, followed by numerical evaluations in Section~\ref{sec:experiment}. Finally, in Section~\ref{sec:conclusion}, we conclude the paper and discuss some future directions.
\section{Problem Formulation and Preliminaries}\label{sec:formulation}
\subsection{Problem Formulation} 
Consider the discrete-time dynamical system
\begin{equation}\label{eq:system}
    \begin{aligned}
        x_{k+1} &= f(x_k,u_k,d), \\
        y_k &= g(x_k,d),
    \end{aligned}
\end{equation}
where at time $k$, $x_k\in \mathbb{R}^{n}$ is the system state, $u_k\in \mathbb{R}^{p}$ is the control input, $y_k\in \mathbb{R}^{q}$ is the measured output, and $d\in \mathbb{R}^{r}$ is the unknown constant exogenous disturbance.
\begin{assumption}\label{assump:system}
    There exists a unique steady-state map $x_{\text{ss}}: \mathbb{R}^{p}\times \mathbb{R}^{r} \to \mathbb{R}^{n}$ such that $\forall u,d, f(x_{\textrm{ss}}(u,d),u,d)=x_{\textrm{ss}}(u,d)$. The map $x_{\textrm{ss}}(u,d)$ is globally $M_x$-Lipschitz in $u$, and the function $g(x,d)$ is globally $M_g$-Lipschitz in $x$. The system \eqref{eq:system} is globally exponentially stable with constant inputs, i.e., there exist $\beta,\tau > 0$ such that for any initial condition $x_0 \in \mathbb{R}^n$, constant input $u_k = u \in \mathbb{R}^{p} (k \in \mathbb{N})$, and $d \in \mathbb{R}^{r}$, the solutions of (1) satisfy $\|x_k - x_{\text{ss}}(u,d)\| \leq \beta\|x_0 - x_{\textrm{ss}}(u,d)\|e^{-\tau k}$.
\end{assumption}

The above properties of the map $x_{\text{ss}}(u,d)$ can be ensured by, e.g., the contraction mapping principle\cite[1A.4]{dontchev2009implicit}. For instance, $x_{\text{ss}}$ is single-valued and Lipschitz continuous in $u$ if for any $d \in \mathbb{R}^r$, $f(x,u,d)$ is globally $M_f$-Lipschitz in $x$ and globally $\widetilde{M}_f$-Lipschitz in $u$, where $M_f \in [0,1), \widetilde{M}_f \in [0,\infty)$. Also, for a linear system $x_{k+1} = Ax_k + Bu_k + Ed$, a Schur stable matrix $A$ (i.e., with a spectral radius less than $1$) suffices to yield these properties\cite{simpson2021analysis}.
Based on Assumption~\ref{assump:system}, in a steady state we have
\begin{equation}\label{eq:ssMap}
    y=g(x_{\text{ss}}(u,d),d)\triangleq h(u,d).
\end{equation}
Additionally, the standard converse Lyapunov theorem (see \cite[p.~194]{khalil2002nonlinear}, \cite[Thm.~6]{kellett2004discrete}, and \cite[Thm.~2]{jiang2002converse}) guarantees that there exist a Lyapunov function $V: \mathbb{R}^{n}\times \mathbb{R}^{p}\times \mathbb{R}^{r} \to \mathbb{R}$ and parameters $\alpha_1, \alpha_2, \alpha_3 > 0$ that rely on $\beta,\tau$ but not on $x_{\text{ss}}(u,d)$ such that
\begin{align}
    &\alpha_1 \|x-x_{\text{ss}}(u,d)\|^2 \leq V(x,u,d) \leq \alpha_2 \|x-x_{\text{ss}}(u,d)\|^2, \label{eq:Vbound} \\
    &V(f(x,u,d),u,d) \!-\! V(x,u,d) \leq -\alpha_3 \|x-x_{\text{ss}}(u,d)\|^2. \label{eq:Vdecrease}
\end{align}
Note that $V$ is relative to the equilibrium $x_{\textrm{ss}}(u,d)$ induced by $u$ and $d$. Indeed, we construct $V(x,u,d) \triangleq \widetilde{V}(x-x_\text{ss}(u,d))$, where $\widetilde{V}$ is the converse Lyapunov function of a shifted autonomous system whose stable equilibrium is the origin.

Based on \eqref{eq:Vbound} and \eqref{eq:Vdecrease}, we can define the rate of change of the function value $V(x_k,u_k,d)$ as
\begin{equation}\label{eq:rate_mu}
    \mu \triangleq \frac{2\alpha_2}{\alpha_1}\left(1-\frac{\alpha_3}{\alpha_2}\right).
\end{equation}
\begin{assumption}\label{assump:systemParam}
    The rate of change $\mu$ satisfies $\mu < 1$.
\end{assumption}

A smaller $\mu$ implies a faster rate at which the system reaches its steady state. For instance, to obtain a satisfactory $\mu$ in sampled-data settings, a feasible approach is to tune the sampling period above a lower bound determined by various factors (e.g., time constants, delays, or non-minimum phase behavior) of plant dynamics\cite[Lemma~2]{belgioioso2021sampled}.
A formal interpretation of this constant will be presented later in Lemma~\ref{lem:expectedLyapunovFuncDecrease}.

Next, we consider the optimization problem
\begin{equation}\label{eq:opt_problem}
    \begin{split}
        \min_{u,y} ~ &\Phi(u,y) \\
        \textrm{s.t.} ~ &y = h(u,d),
    \end{split}
\end{equation}
where $\Phi(u,y)$ is a nonconvex objective function, and $y=h(u,d)$ is the steady-state map \eqref{eq:ssMap}. By eliminating the variable $y$, we transform problem \eqref{eq:opt_problem} to an unconstrained optimization problem in the control input, i.e.,
\begin{equation}\label{eq:opt_problem_reformulated} 
    \min_{u} ~ \obj(u),
\end{equation}
where $\obj(u) \triangleq \Phi(u,h(u,d))$.
\begin{assumption}\label{assump:objective}
    The function $\obj(u)$ is $M$-Lipschitz, and $\inf_{u\in \mathbb{R}^p} \obj(u) > -\infty$. The function $\Phi(u,y)$ is $M_\Phi$-Lipschitz with respect to $y$.
\end{assumption}

The requirements on Lipschitz continuity in Assumption~\ref{assump:objective} are common and largely satisfied in applications. 

Standard numerical optimization\cite{nesterov2018lectures} requires the exact knowledge of the map $h$ and the disturbance $d$ to solve \eqref{eq:opt_problem} or, equivalently, \eqref{eq:opt_problem_reformulated}. As discussed in Section~\ref{sec:introduction}, this approach may be impractical in many applications. In contrast, in this paper, we aim to design a model-free feedback controller that utilizes real-time measurements of $y$ to drive the system \eqref{eq:system} to efficient steady-state operating conditions, as defined by problem \eqref{eq:opt_problem}. The challenges come from the nonconvexity of the objective function, the dynamics of the nonlinear plant, and the goal of being fully model-free.

\subsection{Preliminaries of Zeroth-Order Optimization}
The key idea of zeroth-order optimization is to utilize function evaluations to construct gradient estimates, thus eliminating the need to access gradients directly. The design of this work is inspired by the gradient estimate based on residual feedback in \cite{zhang2020improving}, which owns comparable performance with the mainstream two-point gradient estimates \cite{duchi2015optimal,shamir2017optimal,nesterov2017random}. However, it is significantly easier to implement in the context of FO, as we will see ahead.

For an objective function $\xi(w): \mathbb{R}^p\to \mathbb{R}$, the gradient estimate proposed in \cite{zhang2020improving} is
\begin{equation}\label{eq:one_point_gradient_estimate}
    \widehat{\nabla} \xi(w_k) = \frac{v_k}{\delta} \big(\xi(w_k + \delta v_k) - \xi(w_{k-1} + \delta v_{k-1})\big),
\end{equation}
where $v_k$ and $v_{k-1}$ are independent and identically distributed (i.i.d.) random vectors drawn from the standard multivariate normal distribution, and $\delta>0$ is a smoothing parameter. Note that the objective value evaluated at the previous time $k-1$ is reused in the current iteration $k$. Hence, for \eqref{eq:one_point_gradient_estimate}, only a single new evaluation of $\xi$ is required in each iteration. In contrast, a standard two-point gradient estimate requires two function evaluations, which would translate to two actuation steps in an FO setting; see also Remark~\ref{rem:gradEstimate}.

According to \cite[Lemma~3.1]{zhang2020improving}, $\widehat{\nabla} \xi(w_k)$ in \eqref{eq:one_point_gradient_estimate} is an unbiased estimate of the gradient of the Gaussian smooth approximation $\xi_{\delta}(w)$ for $\xi(w)$ at $w_k$, where
\begin{equation}\label{eq:smoothApproxGaussian}
    \xi_{\delta}(w) = \E_{v\sim \mathcal{N}(0,I)}[\xi(w+\delta v)].
\end{equation}
In \eqref{eq:smoothApproxGaussian}, $\E_{v}[\cdot] \triangleq \int_{\Omega} \cdot\, \ud P_v$ is the expectation with respect to the probability measure $P_v$ of $v$.
We summarize the properties of $\xi_{\delta}(w)$ as follows.
\begin{lemma}[\hspace{1sp}\cite{nesterov2017random}]\label{lem:gaussianSmoothApprox}
    If $\xi:\mathbb{R}^p \to \mathbb{R}$ is $M_\xi$-Lipschitz, then for any $w\in \mathbb{R}^p$, $\delta>0$, and $\xi_{\delta}(w)$ given by \eqref{eq:smoothApproxGaussian},
    \begin{equation*}
        |\xi_{\delta}(w) - \xi(w)|\leq \delta M_\xi \sqrt{p},
    \end{equation*}
    and $\xi_{\delta}(w)$ is $\frac{M_\xi \sqrt{p}}{\delta}$-smooth, i.e., its gradients are $\frac{M_\xi \sqrt{p}}{\delta}$-Lipschitz continuous.
\end{lemma}
\section{Design of Model-Free Feedback Optimization}\label{sec:design}
The proposed controller iteratively updates control inputs along descent steps for the objective function $\obj(u)$. These steps are given by the negative gradient estimates constructed from the evaluations of $\obj(u)$. Such a design does not involve the sensitivity $\nabla_u h(u,d)$ of the steady-state input-output map $h(u,d)$ and therefore makes this controller inherently model-free. An illustration of the interconnection of the system \eqref{eq:system} and this controller is provided in Fig.~\ref{fig:structure}.

\begin{figure}[!tb]
    \centering
    \includegraphics[width=\columnwidth]{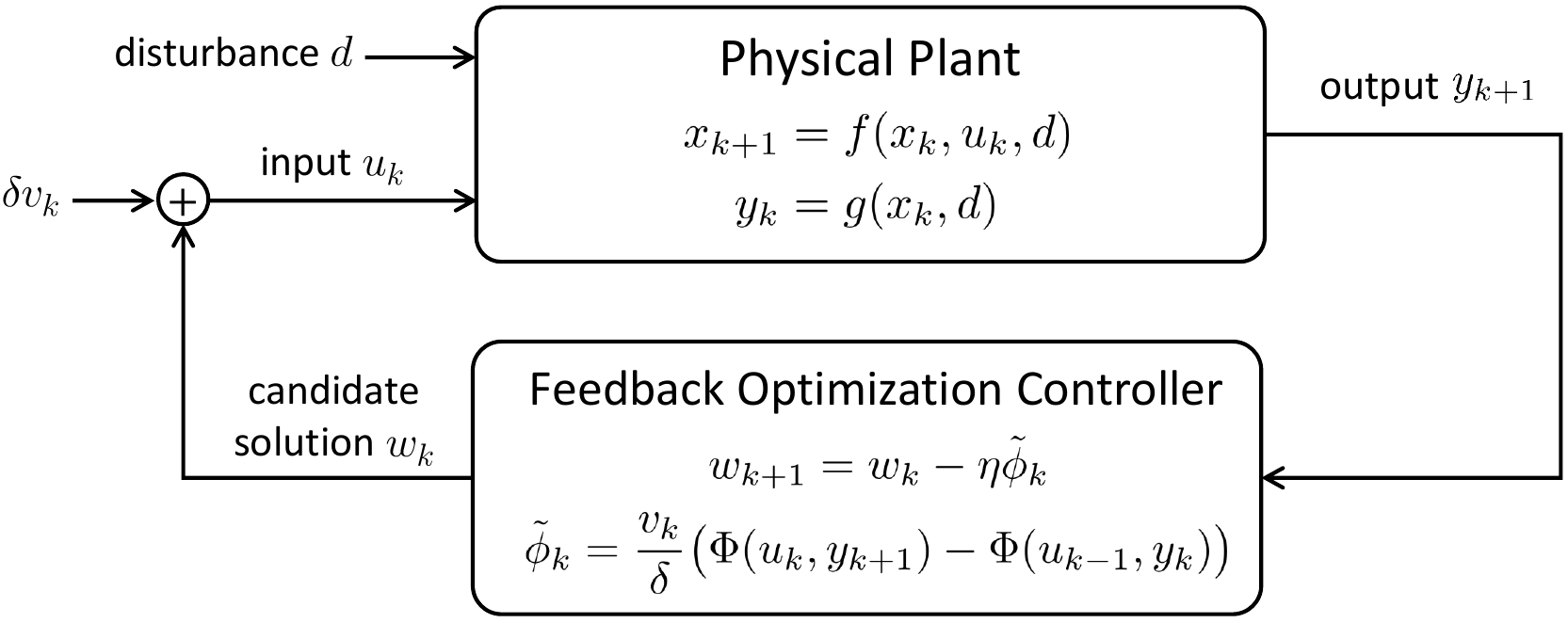}
    \caption{An illustration of the interconnection of the physical plant and the model-free feedback optimization controller.}
    \label{fig:structure}
\end{figure}

The proposed feedback optimization controller is
\begin{subequations}\label{eq:updateOPT}
\begin{align}
        w_{k+1} &= w_k - \eta \tilde{\phi}_{k}, \label{eq:updateDT} \\
        \tilde{\phi}_{k} &= \frac{v_k}{\delta} \big(\Phi(u_k, y_{k+1}) - \Phi(u_{k-1},y_k)\big), \label{eq:updateDTestimate} \\
        u_{k+1} &= w_{k+1} + \delta v_{k+1}, \label{eq:updatePerturbation}
\end{align}
\end{subequations}
where $\eta > 0$ is the step size, $v_0,\ldots,v_{k+1}$ are i.i.d.~standard normal random vectors, and $\delta$ is a smoothing parameter as in \eqref{eq:one_point_gradient_estimate}. The update \eqref{eq:updateOPT} describes the actions performed by the controller at time $k+1$. First, a new candidate solution $w_{k+1}$ is computed based on $\tilde \phi_{k}$, which is the gradient estimate at the previous candidate solution $w_k$. This estimate is constructed once the measurement $y_{k+1}$ is available, since $y_{k+1}$ reflects the influence of $u_k$ and is utilized as a substitute for the steady-state output $h(u_k,d)$. Notice that the historic evaluation $\Phi(u_{k-1},y_k)$ is reused at time $k+1$. Finally, the new input $u_{k+1}$ is obtained by perturbing the candidate solution $w_{k+1}$ with the exploration noise $\delta v_{k+1}$, and it is applied to the system \eqref{eq:system}. The role of the perturbation is to explore \eqref{eq:system} around $w_k$, thus contributing to the effectiveness of $\tilde{\phi}_k$ as a gradient estimate. 



\begin{remark}\label{rem:gradEstimate}
    The gradient estimate \eqref{eq:updateDTestimate} is inspired by the estimate based on residual feedback in \cite{zhang2020improving}. Nonetheless, there exists a key difference in that it uses the approximate objective value $\Phi(u_k,y_{k+1})$ based on real-time measurements of outputs, rather than the exact objective value $\obj(u_k) = \Phi(u_k,h(u_k,d))$. This feature results from the lack of model information (i.e., the map $h$ and the disturbance $d$) and the consideration of system dynamics, which prevents us from instantly accessing steady-state outputs. The inability to evaluate the objective $\Phi$ at steady state poses additional challenges for the closed-loop performance analysis and will be addressed in the following section. 
\end{remark}

\section{Performance Analysis}\label{sec:analysis} 
In this section, we analyze the performance of implementing the proposed feedback optimization controller \eqref{eq:updateOPT} in closed loop with the system \eqref{eq:system}. We first establish some supporting lemmas on the bounds and recursive inequalities of certain important variables. Then, we present our performance certificates and discuss their implications.

\subsection{Supporting Lemmas}
First, we provide an upper bound for the error in the evaluation of $\Phi(u,h(u,d))$. 
Such an error results from the substitution of the measurement $y_{k+1}$ for the steady-state output $h(u_k,d)$ and is defined as
\begin{equation*}
    e_{\Phi}(x_k,u_k) \triangleq \Phi(u_k,y_{k+1}) - \Phi(u_k,h(u_k,d)).
\end{equation*}
\begin{lemma}\label{lem:objErrorSubstitution}
    If Assumptions \ref{assump:system} and \ref{assump:objective} hold, then
    \begin{equation}\label{eq:objSquareErrorSubstitution}
        |e_{\Phi}(x_k,u_k)|^2 \leq \frac{\mu M_\Phi^2 M_g^2}{2\alpha_2} V(x_k,u_k,d).
    \end{equation}
\end{lemma}
\begin{proof}
    Please see Appendix~\ref{subsec:objErrorSubstitution}.
\end{proof}
The right-hand side of \eqref{eq:objSquareErrorSubstitution} involves the Lyapunov function $V$ and the parameters $\alpha_1,\alpha_2,\alpha_3$ corresponding to \eqref{eq:system}. It reflects the rate at which the system reaches the steady state and quantifies the closeness between $\Phi(u_k,y_{k+1})$ and $\obj(u_k)$.

In the following lemmas, we establish the recursive inequalities of two key quantities. One is the expected value of the Lyapunov function, i.e., $\E_{v_{[k]}}[V(x_k,u_k,d)]$, which measures how close the current state $x_k$ is to the steady state $x_{\text{ss}}(u_k,d)$. The other is the second moment of the gradient estimate, i.e., $\E_{v_{[k]}}[\|\tilde{\phi}_k\|^2]$, which reflects the stationary condition at the solution $w_k$. Note that $v_{[k]} \triangleq (v_0,\ldots,v_k)$ is a collection of i.i.d.~samples and $\E_{v_{[k]}}[\cdot] \triangleq \int \cdots \int_{\Omega} \cdot\, \ud P_{v_0}\cdots \ud P_{v_{k}}$ is the expectation with respect to the probability measure of $v_{[k]}$. Also, recall that $u_k$ and $v_k$ are $p$-dimensional vectors.
\begin{lemma}\label{lem:expectedLyapunovFuncDecrease}
    If Assumptions \ref{assump:system} and \ref{assump:objective} hold, with \eqref{eq:updateOPT}, we have
    \begin{align}   
        \E_{v_{[k]}}&[V(x_k,u_k,d)] \notag \\
                 &\leq \mu \E_{v_{[k]}}[V(x_{k-1},u_{k-1},d)] \notag \\ 
                 &\quad + 4\alpha_2\eta^2M_x^2 \E_{v_{[k]}}[\|\tilde{\phi}_{k-1}\|^2] + 8\alpha_2p\delta^2M_x^2. \label{eq:VdecreasePerStepExpect}
    \end{align}
\end{lemma}
\begin{proof}
    Please see Appendix~\ref{subsec:expectedLyapunovFuncDecrease}.
\end{proof}


\begin{lemma}\label{lem:secondMomentDecrease}
    If Assumptions \ref{assump:system} and \ref{assump:objective} hold, with \eqref{eq:updateOPT}, we have
    \begin{align}
        \E_{v_{[k]}}&[\|\tilde{\phi}_k\|^2] \notag \\
                &\leq \frac{6}{\delta^2}M^2\eta^2 p \E_{v_{[k]}}[\|\tilde{\phi}_{k-1}\|^2] + 24M^2(p+4)^2 \notag \\
                &\quad + \frac{3\mu M_\Phi^2 M_g^2}{2\alpha_2 \delta^2} \big((p+4)\E_{v_{[k]}}[V(x_k,u_k,d)] \notag \\
                &\quad + p\E_{v_{[k]}}[V(x_{k-1},u_{k-1},d)]\big). \label{eq:secondMomentDecay}
    \end{align}
\end{lemma}
\begin{proof}
    Please see Appendix~\ref{subsec:secondMomentDecrease}.
\end{proof}

We compactly write the coupled evolution \eqref{eq:VdecreasePerStepExpect} and \eqref{eq:secondMomentDecay} as follows. Let $\pi_k \triangleq \big[\E_{v_{[k]}}[\|\tilde{\phi}_k\|^2]; \E_{v_{[k]}}[t V(x_k,u_k,d)]\big]^{\top}$, where $t>0$ a specific constant. Then, for any $k\in \mathbb{N}_+, \pi_k \preceq C' \pi_{k-1} + \hat{\tau}$, where the elements of $C'$ are
 \begin{equation}\label{eq:contractMatrix}
 \begin{split}
     c'_{11} &= \frac{6\eta^2}{\delta^2}\left(M^2p + M_\Phi^2M_g^2M_x^2(p+4)\mu\right), \quad c'_{22} = \mu, \\
     c'_{12} &= c'_{21} = \frac{\eta M_{\Phi}M_g M_x}{\delta} \sqrt{6\mu(p+(p+4)\mu)},
 \end{split}
 \end{equation}
 and more details on $t,\hat{\tau}$ are provided in \eqref{eq:parametersDecay} and \eqref{eq:parametersDecayTransformed} in Appendix~\ref{subsec:optimality}. Based on the above lemmas, we can bound the partial sums of these coupled sequences and then quantify the optimality of the closed-loop behavior. 


\subsection{Performance Certificate}
We characterize the performance of the interconnection of the controller \eqref{eq:updateOPT} and the system \eqref{eq:system}.
Recall that $\obj_{\delta}$ is the Gaussian smooth approximation for $\obj$ (see \eqref{eq:smoothApproxGaussian}), where $\delta$ is a smoothing parameter. Let $\epsilon>0$ be a given desired precision of this approximation, and let us choose $\delta = \frac{\epsilon}{M\sqrt{p}}$ so that Lemma~\ref{lem:gaussianSmoothApprox} ensures that $|\obj_{\delta}(u) - \obj(u)|\leq \epsilon$\footnote{Ideally, any $\delta\in \big(0,\frac{\epsilon}{M\sqrt{p}}\big]$ would yield the desired precision. In practice, however, the largest possible $\delta$ makes the gradient estimate \eqref{eq:updateDTestimate} more robust to possible measurement noise\cite{liu2020primer}.}.
Additionally, let $T\in \mathbb{N}_{+}$ be a number of iterations set beforehand, and recall that $p$ is the size of the input. The following theorem shows how to choose the constant step size $\eta$ based on $\epsilon$, $T$, and $p$, and what performance guarantee follows from these choices.

\begin{theorem}\label{thm:optimality}
    Suppose that Assumptions \ref{assump:system}-\ref{assump:objective} hold. For a precision $\epsilon > 0$, let $\delta = \frac{\epsilon}{M\sqrt{p}}$, and let $0 < \eta < \frac{\kappa^* \sqrt{\epsilon}}{p^{\frac{3}{2}}\sqrt{T}}$, where
    \begin{equation*}
        \kappa^* = \bigO{\min\Big\{\frac{1-\mu}{\sqrt{\mu}},1-\sqrt{\mu}\Big\}\cdot \sqrt{pT\epsilon}}
    \end{equation*}
    is the upper bound determined by \eqref{eq:kappa2Bound} in Appendix~\ref{subsec:optimality}. The closed-loop interconnection of \eqref{eq:system} and \eqref{eq:updateOPT} ensures that
    \begin{equation}\label{eq:boundOrderComplexity}
        \begin{split}
            \frac{1}{T} &\sum_{k=0}^{T-1} \E_{v_{[T-1]}}[\|\nabla \obj_{\delta}(w_{k})\|^{2}] = \\
                &\bigO{\frac{p^\frac{3}{2}}{\sqrt{\epsilon T}(1-\rho)}} + \bigO{\frac{p^2\sqrt{\mu}}{1-\rho} \cdot \Big(1+\frac{1}{T\epsilon^2 \alpha_2}\Big)},
        \end{split}
    \end{equation}
    where $\rho\in (0,1)$ is the spectral radius of the matrix $C'$ given by \eqref{eq:contractMatrix}. Moreover,
    \begin{equation*}
        \rho = \bigO{\max\Big(\frac{1}{Tp\epsilon},\mu\Big) + \Big(\frac{\mu}{Tp\epsilon}\Big)^{\frac{1}{2}}}.
    \end{equation*}
\end{theorem}
\begin{proof}
    Please see Appendix~\ref{subsec:optimality}.
\end{proof}

In Theorem~\ref{thm:optimality}, the convergence measure is the average second moments of the gradients of the Gaussian smooth approximation $\obj_{\delta}(u)$. This measure reflects first-order stationarity, and we use it as a starting point of analysis, given that finding globally optimal solutions of general nonconvex problems is NP-hard\cite{jin2021nonconvex}.
This measure is common in the field of zeroth-order nonconvex optimization \cite{ghadimi2013stochastic,nesterov2017random,zhang2020improving}, where nonconvexity inhibits the investigation of optimality gaps, and stochasticity required by exploration noises in gradient estimates leads to the consideration of \emph{ergodic} rates.

In the right-hand side of \eqref{eq:boundOrderComplexity}, the first term reflects the order of optimality concerning the descent-based update of the proposed controller, whereas the second term corresponds to the order of errors resulting from the relatively slow response of the system excited by the controller \eqref{eq:updateOPT}. If the system is rapidly decaying, i.e., $\mu$ is very close to $0$, then the aforementioned second term can be neglected, and
\begin{equation}\label{eq:boundOrderComplexitySmallmu}
    \frac{1}{T} \sum_{k=0}^{T-1} \E_{v_{[T-1]}}[\|\nabla \obj_{\delta}(w_k)\|^2] = \mathcal{O}\bigg(\frac{p^\frac{3}{2}}{\sqrt{\epsilon T}(1-\rho)}\bigg),
\end{equation}
where the order of $\rho$ is $\bigO{1/Tp\epsilon}$ in this case. Note that the right-hand side of \eqref{eq:boundOrderComplexitySmallmu} approaches $0$ as $T \to \infty$.

Moreover, in the transient stage of iterations (i.e. at time $k=0,\ldots,T-1$), $w_k$ comprises the set of candidate solutions, while $u_k$ is the actual input obtained by perturbing $w_k$ (see \eqref{eq:updatePerturbation}) and fed into the system. Theorem~\ref{thm:optimality} characterizes optimality in terms of $w_k$. In the following corollary, we quantify sub-optimality in terms of $u_k$ aligned with the objective \eqref{eq:opt_problem_reformulated}.
\begin{corollary}\label{cor:suboptimalityCor}
    If the conditions of Theorem~\ref{thm:optimality} hold, then
    \begin{equation}\label{eq:boundExpectedTransientSecondMoment}
    \begin{split}
        \frac{1}{T}\sum_{k=0}^{T-1} &\E_{v_{[T-1]}}[\|\nabla \obj_{\delta}(u_k)\|^{2}] \\
            &\leq \frac{2}{T}\sum_{k=0}^{T-1}\E_{v_{[T-1]}}[\|\nabla \obj_{\delta}(w_k)\|^2] + 2M^2p^2.
    \end{split}
    \end{equation}
\end{corollary}
\begin{proof}
    Please see Appendix~\ref{subsec:suboptimalityCor}.
\end{proof}

Finally, we characterize the closed-loop stability properties.
\begin{theorem}\label{thm:stability}
    If the conditions of Theorem~\ref{thm:optimality} hold, then,
    \begin{equation}\label{eq:boundOrderStateDist}
        \begin{split}
            \frac{1}{T} &\sum_{k=0}^{T-1} \E_{v_{[T-1]}} [\|x_{k+1} - x_{\textrm{ss}}(u_k,d)\|^2] \\
                    &= \bigO{\frac{\epsilon^{\frac{3}{2}}\sqrt{\mu}}{\sqrt{pT}(1-\rho)}} + \bigO{\frac{\big(\epsilon^2+\frac{1}{T\alpha_2}\big)\mu}{1-\rho}}.
        \end{split}
    \end{equation}
\end{theorem}
\begin{proof}
    Please see Appendix~\ref{subsec:stability}.
\end{proof}

In Theorem~\ref{thm:stability}, we focus on the average second moments of the distances between the state $x_{k+1}$ after $u_k$ is applied at time $k$ and the steady state $x_{\text{ss}}(u_k,d)$. Theorem~\ref{thm:stability} suggests that the closed-loop stability properties heavily depend on the rate of change $\mu$ defined in \eqref{eq:rate_mu}. In fact, for a rapidly decaying system (i.e., with a $\mu$ very close to $0$), the right-hand side of \eqref{eq:boundOrderStateDist} is also approximately $0$.

\begin{remark}
    If $V(x,u,d)$ is $M_V$-Lipschitz in $u$, then we can establish closed-loop guarantees without Assumption~\ref{assump:systemParam}. To this end, we can follow similar steps and derive an alternative version of Lemma~\ref{lem:expectedLyapunovFuncDecrease}, where the coefficients of the last two terms in \eqref{eq:VdecreasePerStepExpect} will involve the Lipschitz constant $M_V$. Notice that a rather large $M_V$ may cause a strict requirement on the step size (of the order of $1/\sqrt{M_V}$) and a slow closed-loop convergence rate. Moreover, the condition of a Lipschitz continuous $V$ may require additional properties of nonlinear dynamics, e.g., a bounded distance to the steady state.
\end{remark}
\section{Extension to Constrained Input}\label{sec:extension}
We now present the extension of the proposed model-free feedback optimization controller to handle input constraints.

A common type of such constraints is represented by input saturation that results from actuation limits or actions of low-level controllers\cite{hauswirth2021optimization}. To tackle these constraints, we can directly implement the proposed controller \eqref{eq:updateOPT} and outsource the constraint enforcement to the physical plant. For instance, consider the following system with the saturated control input
\begin{equation}\label{eq:system_saturated}
    \begin{aligned}
        x_{k+1} &= f(x_k,\sat(u_k),d) \triangleq \tilde{f}(x_k,u_k,d), \\
        y_k &= g(x_k,d),
    \end{aligned}
\end{equation}
where $\sat(u_k)$ equals to $u_k$ if $u_k$ falls in the constraint set $\mathcal{U}$, and it equals to some limit value otherwise. Since $\sat(\cdot)$ will not change the Lipschitz continuity of the corresponding steady-state map and the objective function, the controller \eqref{eq:updateOPT} can still drive the system \eqref{eq:system_saturated} to its efficient operating point, while the constraint $\mathcal{U}$ is naturally enforced by the physical plant through saturation effects.

In the remainder of this section, we focus on coping with other general engineering constraints that may couple multiple inputs, are not enforced by low-level saturation, and therefore need to be accounted for in the controller design.

\subsection{Problem Reformulation and Preliminaries}
Let the input constraint set be denoted by $\mathcal{U}\subset \mathbb{R}^p$. In this case, the optimization problem becomes
\begin{equation}\label{eq:opt_problem_cons}
    \begin{split}
        \min_{u,y} ~ &\Phi(u,y) \\
        \textrm{s.t.} ~ &y = h(u,d), \\
            &u \in \mathcal{U}.
    \end{split}
\end{equation}
Problem \eqref{eq:opt_problem_cons} can be equivalently reformulated as
\begin{equation}\label{eq:infiObjCons}
    \min_{u} ~ \obj(u), \qquad \textrm{s.t.} ~ u \in \mathcal{U},
\end{equation}
where $\obj(u) \triangleq \Phi(u,h(u,d))$. We accordingly modify the requirement on the finite infimum in Assumption~\ref{assump:objective} to
\begin{equation}
    \inf_{u\in \mathcal{U}} \obj(u) > -\infty.
\end{equation}
The assumption on $\mathcal{U}$ is given as follows.
\begin{assumption}\label{assump:constraintSet}
    The set $\mathcal{U}$ is convex, closed, and bounded with diameter $D$, i.e., $\|u_1-u_2\|\leq D$ for any $u_1,u_2 \in \mathcal{U}$.
\end{assumption}
In applications, there are usually upper and lower bounds that define the admissible input\cite{hauswirth2021optimization}, which translates to the bounded constraint set. Notice that $D$ is not involved in the updates of the controller and is only used in the performance analysis. The existence of $\mathcal{U}$ renders the standard normal random vectors $v_k$ utilized Section~\ref{sec:design} unfavorable, since the unboundedness of Gaussian distributions may cause the perturbed input $u_k = w_k + \delta v_k$ to lie far outside $\mathcal{U}$. Instead, we uniformly sample $v_k$ from the unit sphere $\mathbb{S}_{p-1}\triangleq \{v\in \mathbb{R}^p:\|v\|=1\}$. This scheme of sampling is also adopted by \cite{shamir2017optimal,gao2018information,chen2020frank}. In this case, the smooth approximation $\obj_{\delta}(w)$ for the objective $\obj(w)$ is
\begin{equation}\label{eq:smoothApproxUniform}
    \obj_{\delta}(w) = \E_{v\sim U(\mathbb{B}_p)}[\obj(w+\delta v)],
\end{equation}
where $U(\mathbb{B}_p)$ is the uniform distribution over the closed unit ball $\mathbb{B}_p$ in $\mathbb{R}^p$. Similar to Lemma~\ref{lem:gaussianSmoothApprox}, the properties of $\obj_{\delta}(w)$ are summarized as follows.
\begin{lemma}\label{lem:uniformSmoothApprox}
    If $\obj: \mathbb{R}^p \to \mathbb{R}$ is $M$-Lipschitz, then for any $w\in \mathbb{R}^p$, $\delta>0$, and $\obj_{\delta}(w)$ defined in \eqref{eq:smoothApproxUniform},
    \begin{subequations}
        \begin{align}
            \E_{v\in U(\mathbb{S}_{p-1})}\left[\frac{p}{\delta}\obj(w+\delta v)v\right] &= \nabla \obj_{\delta}(w), \label{eq:expectedGradSmoothApprox} \\
            \big|\obj_{\delta}(w) - \obj(w)\big| &\leq M \delta, \label{eq:closenessSmoothApprox}
        \end{align}
    \end{subequations}
    and $\obj_{\delta}(w)$ is $\frac{Mp}{\delta}$-smooth, i.e., its gradients are $\frac{Mp}{\delta}$-Lipschitz continuous.
\end{lemma}
\begin{proof}
    The equation \eqref{eq:expectedGradSmoothApprox} is proved in \cite[Lemma~2.1]{flaxman2005online}. The proofs of \eqref{eq:closenessSmoothApprox} and the smoothness of $\obj_{\delta}(w)$ are similar to those of \cite[Lemma~4.1]{gao2018information}, where $\obj(w)$ is assumed to be a smooth function. 
\end{proof}

\begin{remark}
    In problem~\eqref{eq:opt_problem_cons}, if there exist further constraints $c_i(y)\leq 0,i=1,\ldots,m$ on $y$, then we can penalize the violation of these output constraints and optimize $\Phi(u,y) + \sigma \sum_{i=1}^m (\max(c_i(y),0))^2$ subject to the remaining constraints, where $\sigma>0$ is a large penalty parameter. In this case, we additionally require that all $c_i(y)$ are Lipschitz continuous in $y$ to ensure that the overall objective satisfies Assumption~\ref{assump:objective}. Moreover, linear (profit-optimizing) terms in the objective are covered by Assumption~\ref{assump:objective} of Lipschitz continuity.
\end{remark}

\subsection{Design of Constrained Feedback Optimization}
In the input-constrained case, the updates of the proposed model-free feedback optimization controller are
\begin{subequations}\label{eq:updateOPTCons}
\begin{align}
        w_{k+1} &= (1-\eta)w_k + \eta s_k, \label{eq:updateDTCons} \\
        s_k &= \argmin_{s\in \mathcal{U}} \langle s,\tilde{\phi}_k \rangle, \label{eq:LMCons} \\
        \tilde{\phi}_k &= \frac{pv_k}{\delta} \big(\Phi(u_k, y_{k+1}) - \Phi(u_{k-1},y_k)\big), \label{eq:updateDTestimateCons} \\
        u_{k+1} &= w_{k+1} + \delta v_{k+1}, \label{eq:updatePerturbationCons}
\end{align}
\end{subequations}
where $\eta \in (0,1)$ is the step size, $\langle\,,\rangle$ denotes the inner product, $v_0,\ldots,v_{k+1}$ are i.i.d.~samples from the uniform distribution $U(\mathbb{S}_{p-1})$ on the unit sphere $\mathbb{S}_{p-1}$, and $\delta$ is a smoothing parameter. The initial point $w_0$ is chosen from the constraint set $\mathcal{U}$. The update \eqref{eq:updateOPTCons} is based on the Frank-Wolfe algorithm for constrained optimization\cite{jaggi2013revisiting}, combined with the residual one-point gradient estimation that we have already seen before in \eqref{eq:updateOPT}. At time $k+1$, the controller calculates a new candidate solution $w_{k+1}$ by taking a convex combination of the previous solution $w_k$ and a new point $s_k$. Note that $s_k$ is obtained in the step of linear minimization constrained over $\mathcal{U}$, where the gradient estimate $\tilde{\phi}_k$ corresponding to $w_k$ is utilized. Similar to \eqref{eq:updateOPT}, $\tilde{\phi}_k$ is constructed when the output measurement $y_{k+1}$ is at hand. Finally, the controller applies the input $u_{k+1}$ perturbed with the exploration noise $\delta v_{k+1}$ to the system \eqref{eq:system}.
\begin{remark}
    Compared to zeroth-order Frank-Wolfe algorithms\cite{balasubramanian2021zeroth,chen2020frank} that use two-point estimates, in \eqref{eq:updateDTestimateCons}, we use the current and the historic evaluations of the objective function to construct the gradient estimate. This design corresponds to the scenario where for a dynamic plant, at every time $k+1$, we can only obtain one new approximate objective value $\Phi(u_k,y_{k+1})$ based on the real-time measurement $y_{k+1}$.
\end{remark}

\subsection{Performance Analysis}
We use the Frank-Wolfe gap, i.e.,
\begin{equation}\label{eq:FWgap}
    \mathcal{G}(w_k) = \max_{w\in \mathcal{U}} \langle w-w_k,-\nabla \obj_{\delta}(w_k) \rangle
\end{equation}
to measure optimality at $w_k$. In \eqref{eq:FWgap}, $\obj_{\delta}$ is the smooth approximation for $\obj$ (see \eqref{eq:smoothApproxUniform}), and $\delta$ is a smoothing parameter. The gap $\mathcal{G}(w_k)$ is non-negative for any $w_k\in \mathcal{U}$ and equals $0$ if and only if $w_k$ is a stationary point of $\obj_{\delta}(w)$, i.e., $\forall w \in \mathcal{U}, \langle \nabla \obj_{\delta}(w_k), w-w_k \rangle \geq 0$ \cite{chen2020frank}.

As before, let $\epsilon > 0$ be the user-specified precision of the smooth approximation, and let us choose $\delta = \frac{\epsilon}{M}$
so that Lemma~\ref{lem:uniformSmoothApprox} ensures that $|\obj_{\delta}(u) - \obj(u)|\leq \epsilon$. Similarly to Theorem~\ref{thm:optimality}, the following theorem demonstrates how to choose the constant step size $\eta$ based on $\epsilon$, the pre-set number of iterations $T$, and the input size $p$, and what performance guarantee follows from these choices.

\begin{theorem}\label{thm:constrainedExtension}
    Suppose that Assumptions~\ref{assump:system}-\ref{assump:constraintSet} hold. Let $\delta= \frac{\epsilon}{M}$, $\eta= \kappa \sqrt{\frac{\epsilon}{\smash{p T}}}$, where $\kappa \in \big(0,\frac{\sqrt{pT}}{\sqrt{\epsilon}}\big)$. For the closed-loop interconnection of \eqref{eq:system} and \eqref{eq:updateOPTCons}, we have
    \begin{align}\label{eq:boundOrderComplexityCons}
        \frac{1}{T}&\sum_{k=1}^{T}\E_{v_{[T]}}[\mathcal{G}(w_k)] \leq \notag \\
        &\Big(\frac{\E_{v_{[T]}}[\obj_{\delta}(w_1)]-\obj_{\delta}^*}{\kappa} + \frac{D^2M^2\kappa}{2}\Big)\sqrt{\frac{p}{\epsilon T}} \notag \\
        &+ D\bigg\{6M^4D^2\kappa^2\cdot\frac{p}{\epsilon T} + 24M^2p^2 + \frac{\mu M_\Phi^2M_g^2}{2\alpha_2} \notag \\
        &\hspace{3em} \cdot \bigg[\frac{1+\mu}{(1-\mu)T}\cdot \Big(E_{v_{[T]}}[V(x_0,u_0,d)] \notag \\
        &\hspace{4em}+ 2(T-1)\alpha_2 M_x^2\big(2\kappa^2D^2 \frac{\epsilon}{pT} + \frac{8}{M^2} \epsilon^2\big)\Big) \notag \\
        &\hspace{4em} + 2\alpha_2M_x^2\big(2\kappa^2D^2 \frac{\epsilon}{pT} + \frac{8}{M^2} \epsilon^2\big) \bigg] \bigg\}^\frac{1}{2},
    \end{align}
    where $\obj_{\delta}^{*} = \inf_{w\in \mathcal{U}} \obj_{\delta}(w) > -\infty$.
\end{theorem}
\begin{proof}
    Please see Appendix~\ref{subsec:constrainedExtension}.
\end{proof}

Note that if the plant is fast-decaying, i.e., $\mu$ (and thus also $\mu/2\alpha_2$) is very close to $0$, then
\begin{equation}\label{eq:boundOrderComplexityConsSimp}
    \frac{1}{T}\sum_{k=1}^{T}\E_{v_{[T]}}[\mathcal{G}(w_k)] \leq C_1 \sqrt{\frac{p}{\epsilon T}} + D\left(C_2\frac{p}{\epsilon T} + 24M^2p^2 \right)^\frac{1}{2},
\end{equation}
where $C_1$ and $C_2$ are constants. As $T\to \infty$, the right-hand side of \eqref{eq:boundOrderComplexityConsSimp} approaches $2\sqrt{6}DMP$. This nonzero upper bound mainly
results from the variance of the gradient estimate \eqref{eq:updateDTestimateCons}. It is the price to pay since we can only obtain one new evaluation of the objective function at every iteration and cannot lower the variance by increasing the sample size. The remaining terms in \eqref{eq:boundOrderComplexityCons} are due to the dynamical plant being persistently excited by the proposed controller.

\begin{remark}
    In Theorem~\ref{thm:constrainedExtension}, we use the average of the expected Frank-Wolfe dual gap as the convergence measure. This ergodic-type measure is similar to that investigated in Theorem~\ref{thm:optimality}. The interpretation is that it equals $\E_{v_{[T]}}[\mathcal{G}(w_R)]$, where $w_R$ is selected uniformly at random from the candidate solutions $\{w_k\}_{k=1}^{T}$. Furthermore, the right-hand side of \eqref{eq:boundOrderComplexityCons} is also the upper bound on $\E_{v_{[T]}}[\min_{k=1,\ldots,T} \mathcal{G}(w_k)]$, i.e., the expected value of the minimum Frank-Wolfe dual gap, as considered in \cite{chen2020frank}. 
    This conclusion is drawn from the inequality $\E_{v_{[T]}}\left[\min_{k=1,\ldots,T} \mathcal{G}(w_k)\right] \leq \E_{v_{[T]}}\big[\frac{1}{T} \sum_{k=1}^{T} \mathcal{G}(w_k)\big]  = \frac{1}{T} \sum_{k=1}^{T} \E_{v_{[T]}}[\mathcal{G}(w_k)]$.
\end{remark}
\section{Numerical Evaluations}\label{sec:experiment}
We illustrate the performance of the proposed feedback controller and compare it with the first-order counterparts (i.e., feedback optimization based on gradient descent).
Consider the following nonlinear system
\begin{equation}\label{eq:sys_simulation}
    \begin{split}
        x_{k+1} &= Ax_k + Bu_k + Ed_x + F(x_k-x_{\text{ss}}(u_k,d_x))^{\otimes 2}, \\
        y_k &= Cx_k + Dd_y,
    \end{split}
\end{equation}
where $x\in \mathbb{R}^{10}$, $u\in \mathbb{R}^{5}$, $d_x,d_y\in \mathbb{R}^5$ and $y\in \mathbb{R}^{5}$ are the state, input, disturbances, and output, respectively. Moreover, $x_{\text{ss}}(u_k,d_x) \triangleq (I-A)^{-1}(Bu_k+Ed_x)$ is the steady-state map, where $I$ is the identity matrix, and $(\cdot)^{\otimes 2}$ denotes the Kronecker product $(\cdot) \otimes (\cdot)$. The last term in the state equation of \eqref{eq:sys_simulation} can be interpreted as a remainder term when general nonlinear dynamics is approximated by linear dynamics at its steady state. The elements of the system matrices in \eqref{eq:sys_simulation} are randomly drawn from the standard uniform distribution. We further scale $A$, $F$ to let $\|A\|_2 = 0.05$ and $\|F\|_1 = 0.01$. The disturbances $d_x$, $d_y$ are generated from the standard normal distribution.
The optimization problem is
\begin{equation}\label{eq:opt_problem_simulation}
    \min_{u,y} ~ \Phi(u,y)= \underbrace{-\lambda\|u\|^3}_{\Phi_1(u)} + \underbrace{u^{\top}M_1u + M_2^{\top}u + \|y\|^2}_{\Phi_2(u,y)} + \underbrace{\lambda'\|u\|_1}_{\Phi_3(u)},
\end{equation}
where $u$ and $y$ are the input and the steady-state output of the system \eqref{eq:sys_simulation}, respectively. In \eqref{eq:opt_problem_simulation}, $\Phi_1(u)$ is a concave function and $\lambda = 10^{-2}$. Let $M_1 = M_3^{\top}M_3 \in \mathbb{R}^{5\times 5}$ be a positive semidefinite matrix. We draw the elements of $M_2\in \mathbb{R}^{5}$ and $M_3\in \mathbb{R}^{5\times 5}$ from the standard uniform distribution. Hence, $\Phi_2(u,y)$ is a smooth convex function. For the nonsmooth regularization $\Phi_3(u)$, we set $\lambda'=10^{-3}$. The $\ell_1$-norm regularization is widely used to promote the sparsity of solutions\cite{metel2021stochastic}. The objective $\Phi(u,y)$ is a nonconvex function of $u$. 

To solve general nonsmooth problems, we can apply the proposed model-free feedback controller \eqref{eq:updateOPT} and the first-order counterpart based on sub-gradient descent. Nonetheless, to exploit the composite feature of problem \eqref{eq:opt_problem_simulation}, for both controllers, we add a proximal operator to the descent-based updates. That is, for the model-based first-order FO as a comparison, we implement
\begin{align}\label{eq:proximalFO}
        u_{k+1} = \prox_{\eta\Phi_3}\big(&u_k-\eta\big(\nabla \Phi_1(u_k) + \nabla_u\Phi_2(u_k,y_{k+1}) \notag \\
            & ~~+ H^{\top}\nabla_y\Phi_2(u_k,y_{k+1})\big)\big),
\end{align}
where $\eta > 0$ is the step size, $H\triangleq C(I-A)^{-1}B$ is the steady-state input-output sensitivity of \eqref{eq:sys_simulation}, and $\prox_{\eta\Phi_3}(u)$ is the proximal operator of $\eta\Phi_3(u) = \eta\lambda'\|u\|_1$, i.e.,
\begin{align*}
    \prox_{\eta\Phi_3}(u) &\triangleq \argmin_{u'\in \mathbb{R}^{10}}\Big(\eta\Phi_3(u') + \frac{1}{2}\|u'-u\|^2\Big) \\ 
            &= \text{sign}(u)\max\{|u|-\eta\lambda',0\}.
\end{align*}
For the proposed controller \eqref{eq:updateOPT}, we modify \eqref{eq:updateDT} to
\begin{equation*}
    w_{k+1} = \prox_{\eta\Phi_3}(w_k-\eta\tilde{\phi}_k).
\end{equation*}
We also implement the following discrete-time stochastic extremum seeking algorithm adapted from\cite{liu2016stochastic}
\begin{subequations}\label{eq:stochasticES}
\begin{align}
    w_{k+1} &= w_k - \eta\sin(v_{k+1}) \Phi(u_k,y_{k+1}), \\
    u_{k+1} &= w_{k+1} + a\sin(v_{k+1}),
\end{align}
\end{subequations}
where $\eta>0$ is the step size, $a>0$ is the amplitude of perturbation, and $v_0,\ldots,v_{k+1}$ are i.i.d.~standard normal random vectors. The code is available at \cite{he2023modelcode}.

Fig.~\ref{fig:comparison} illustrates the performance of the system \eqref{eq:sys_simulation} interconnected with different controllers. We focus on the squared norm of the subgradient of $\obj(u)$, i.e., $\|\partial \obj(u)\|^2$. For the proposed controller, we set the smoothing parameter $\delta = 5\times 10^{-5}$ and the step size $\eta = 2.5\times 10^{-5}$ and run $20$ independent experiments. The step size $\eta$ is obtained in a bisection manner until divergence occurs. The quantitative bound on $\eta$ in Theorem~\ref{thm:optimality} is conservative, but it offers a qualitative guideline for tuning. We consider three realizations of the first-order controller \eqref{eq:proximalFO}, i.e., with the exact sensitivity $H$, an inexact $\hat{H}$, and sensitivity estimation\cite{picallo2021adaptive} that starts from $\hat{H}$ and uses excitation noises sampled from the normal distribution $\mathcal{N}(0,\sigma_u^2I)$. Specifically, $\hat{H}$ is a randomly perturbed version of $H$ with a $15\%$ range of relative errors in elements, and $\sigma_u = 10^{-6}$. The step sizes used in these realizations are all $10^{-4}$. For the stochastic ES \eqref{eq:stochasticES}, we set $\eta = 10^{-4}$ and $a = 0.05$.

\begin{figure}[!tb]
    \centering
    \includegraphics[width=0.71\columnwidth]{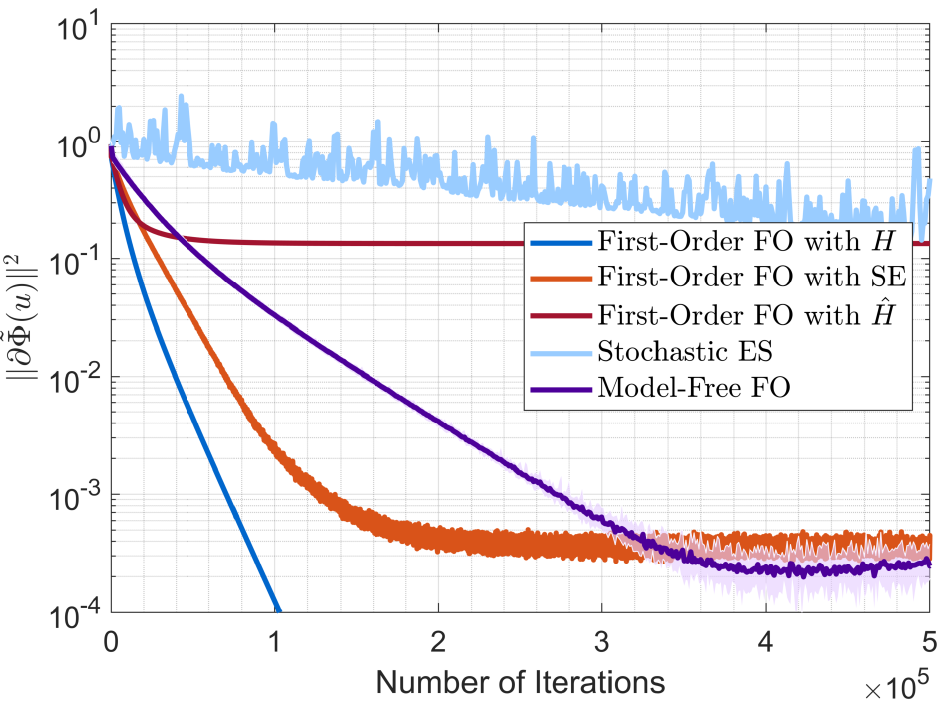}
    \caption{Comparison of the proposed model-free feedback controller, the first-order counterparts, and a stochastic extremum seeking algorithm. Note that SE represents sensitivity estimation.}
    \label{fig:comparison}
\end{figure}

In Fig.~\ref{fig:comparison}, the solid curve corresponding to the model-free FO represents the average trajectory of all the experiments, and the shaded area indicates the ranges of changes of those trajectories. The oscillations in this curve stem from the average effect of the stochasticity in the input perturbation. This phenomenon does not exist for the first-order FO with $H$ and $\hat{H}$ because of their deterministic update rules. We observe that the stochastic ES converges relatively slowly. The first-order feedback controller based on the exact sensitivity $H$ yields the best performance. Nonetheless, an inexact estimation $H$ leads to severe degradation in convergence rates and solution accuracy. Sensitivity estimation addresses these issues, although a gap still exists compared to FO with the exact $H$. The proposed controller achieves a comparable solution accuracy as FO with sensitivity estimation. Moreover, it is easy to implement, without requiring prior model information or operations of estimation, and it is applicable even when the steady-state map is non-differentiable.

Fig.~\ref{fig:comparisonZFO} demonstrates the performance of the proposed feedback controller when different $\delta$ and $\eta$ are used. We observe that a smaller $\delta$ leads to a narrower range of changes in trajectories and better solution accuracy. Note that $\delta$ cannot be set as an arbitrarily small number, since in that case the random direction (see \eqref{eq:updateDTestimate}) will be dominated by the measurement noises in applications. Additionally, if falling in the range that guarantees convergence, a larger step size $\eta$ helps to accelerate the rate of convergence to optimality.
\begin{figure}[!tb]
    \centering
    \includegraphics[width=0.71\columnwidth]{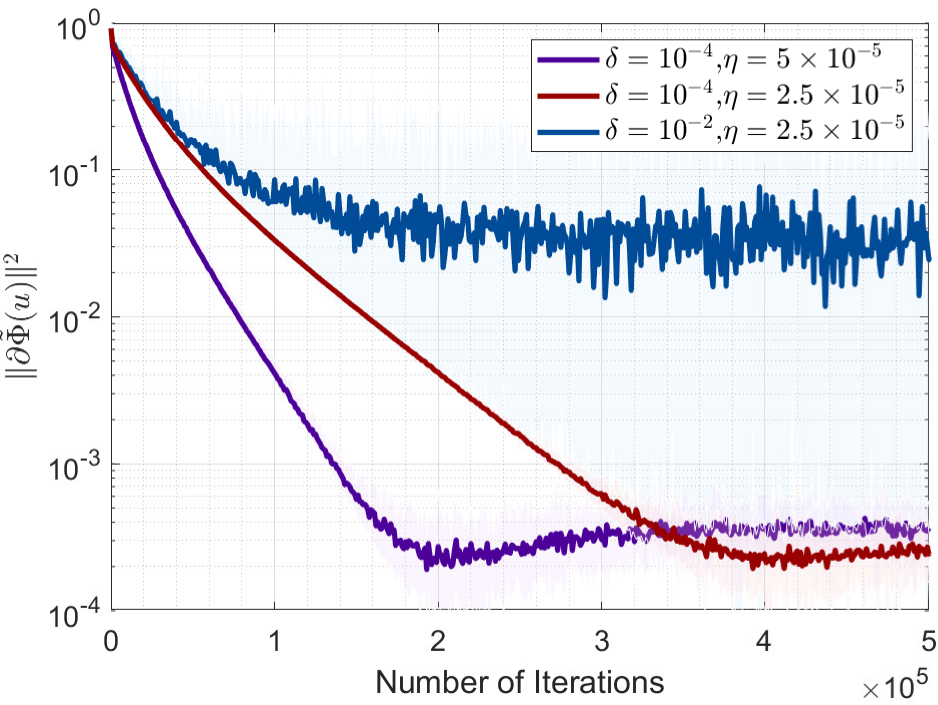}
    \caption{Performance of the proposed model-free feedback controller with different $\delta$ and $\eta$.}
    \label{fig:comparisonZFO}
\end{figure}

Next, we consider the extension to constrained inputs. In this case, we optimize the objective function $\Phi(u,y)$ in \eqref{eq:opt_problem_simulation} over $\mathcal{U}=\{u|\underline{u} \preceq u \preceq \bar{u}\}$, where $-\underline{u}$ and $\bar{u}$ are drawn from the standard uniform distribution. We center on the Frank-Wolfe gap of $\obj(u)$. For this problem that involves the box constraint $\mathcal{U}$ and the nonsmooth regularization, the proximal versions of the descent methods are relatively more involved (see \cite{metel2021stochastic}). Hence, for the mentioned three realizations of first-order FO, we implement the projected sub-gradient descent
\begin{equation*}
    \begin{split}
        u_{k+1} = \Pi_{\mathcal{U}}\big[&u_k-\eta\big(\nabla \Phi_1(u_k) + \nabla_u\Phi_2(u_k,y_{k+1}) \\
            & ~~+ H^{\top}\nabla_y\Phi_2(u_k,y_{k+1}) + \partial \Phi_3(u_k)\big)\big],
    \end{split}
\end{equation*}
where $\Pi_{\mathcal{U}}[\cdot]$ is the projection onto $\mathcal{U}$, and $\partial \Phi_3(u_k)$ is the sub-gradient of $\Phi_3$ at $u_k$. Thanks to the problem setup, both $\Pi_{\mathcal{U}}[\cdot]$ and $\partial \Phi_3(u)$ have closed-form expressions. We set $\sigma_u = 10^{-5}$ and tune the step sizes to be $10^{-4}$ in all these three realizations. For model-free FO, we implement both the Frank-Wolfe-type update \eqref{eq:updateOPTCons} and additionally (for comparison) the projected version of the descent-based update \eqref{eq:updateDT}, i.e.,
\begin{equation}\label{eq:updateDTproj}
    w_{k+1} = \Pi_{\mathcal{U}}[w_k-\eta\tilde{\phi}_k].
\end{equation}
We set the smoothing parameter $\delta$, the step size $\eta$, and the number of experiments as $10^{-3}$, $10^{-5}$, and $20$, respectively. The initial points are chosen to be $(\bar{u}+\underline{u})/2$ in all the aforementioned implementations. Fig.~\ref{fig:constrained} illustrates the results of convergence. We observe that the error corresponding to the model-free controller with Frank-Wolfe-type updates is relatively large. The possible reason is that the performance of such a controller is more sensitive to the non-vanishing variance of the gradient estimate \eqref{eq:updateDTestimate}, and further investigations are needed to improve this practical performance. For the model-free controller based on projected descent \eqref{eq:updateDTproj}, the solution accuracy is better than that of the first-order feedback controller with sensitivity estimation.

\begin{figure}[!tb]
    \centering
    \includegraphics[width=0.71\columnwidth]{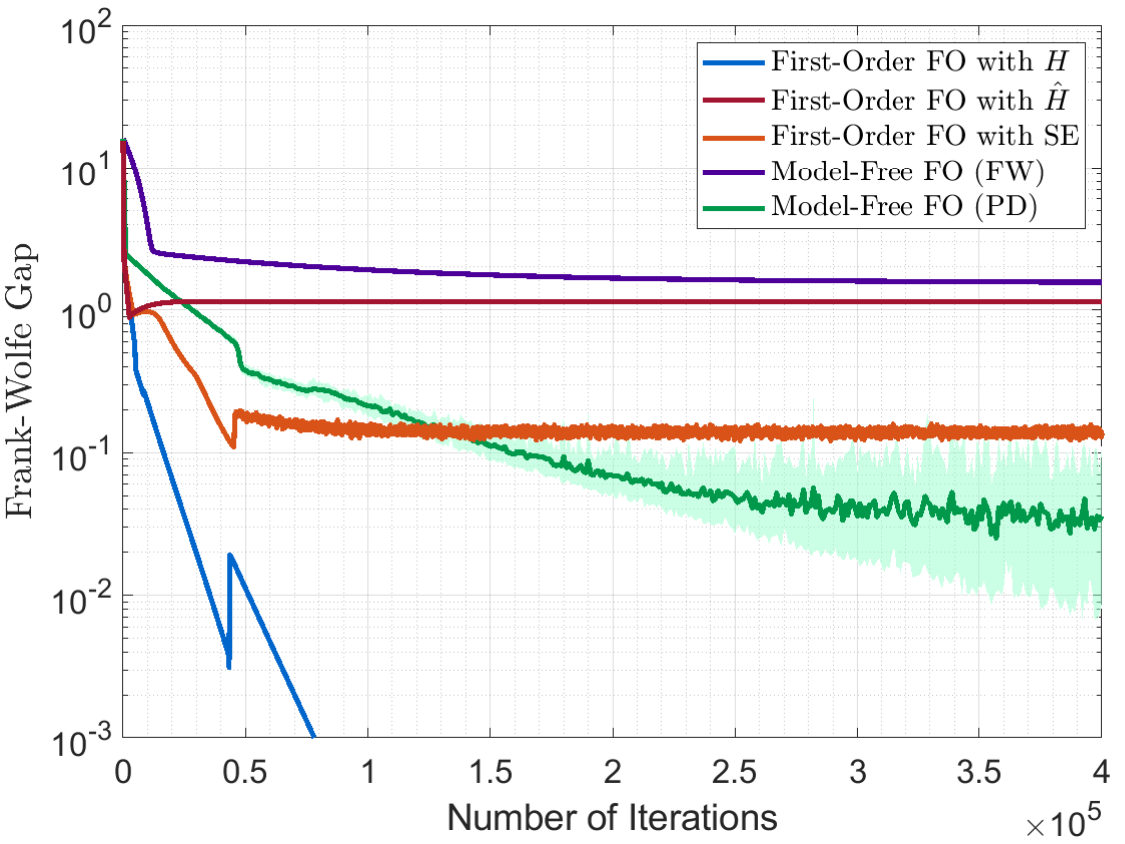}
    \caption{Comparison between the proposed model-free feedback controller and the first-order counterparts for solving the problem with the input constraint set. FW and PD refer to model-free FO based on Frank-Wolfe-type updates \eqref{eq:updateOPTCons} and projected descent \eqref{eq:updateDTproj}, respectively.}
    \label{fig:constrained}
\end{figure}

We now extend the main results a bit further by illustrating the tracking performance of different methods when the system \eqref{eq:sys_simulation} involves the time-varying disturbance $d_k$. We consider the following unconstrained convex problem 
\begin{equation}\label{eq:opt_problem_tracking}
    \min_{u,y} ~ u^{\top}M_1u + M_2^{\top}u + \|y\|^2 + \lambda'\|u\|_1,
\end{equation}
where $u$ and $y$ are the input and the steady-state output of the system \eqref{eq:sys_simulation}, respectively. For every $d_k$, problem~\eqref{eq:opt_problem_tracking} owns a single pair of the optimal solution $u_k^*$ and optimal value $\Phi_k^*$.
We focus on the tracking performance in terms of $u_k^*$ and $\Phi_k^*$, which are measured by the error $\|u_k-u_k^*\|$ and the optimality gap $\tilde{\Phi}(u_k)-\Phi_k^*$, respectively. Every $5\times 10^3$ iteration, the disturbances are re-drawn from the uniform distribution $U(10^{-3},10^{-3})$. For the first-order feedback optimization controllers with the exact $H$, inexact $\hat{H}$, and sensitivity estimation, the step sizes are all $10^{-4}$. For the proposed controller, we set $\delta = 10^{-4}$ and $\eta = 10^{-4}$. As shown in Fig.~\ref{fig:trackingResult}, similar to the previous case studies, the first-order controller with the exact sensitivity enjoys the best performance. Compared to the counterpart with an inexact sensitivity, the proposed controller achieves a closer tracking effect, though it features a slower convergence rate. Additionally, the feedback nature of these controllers allows the autonomous suppression of the spikes resulting from the change of the disturbance $d_k$.
\begin{figure}[!tb]
    \centering
    \subfloat[Error $\|u_k-u_k^*\|$.]{\includegraphics[width=0.71\columnwidth]{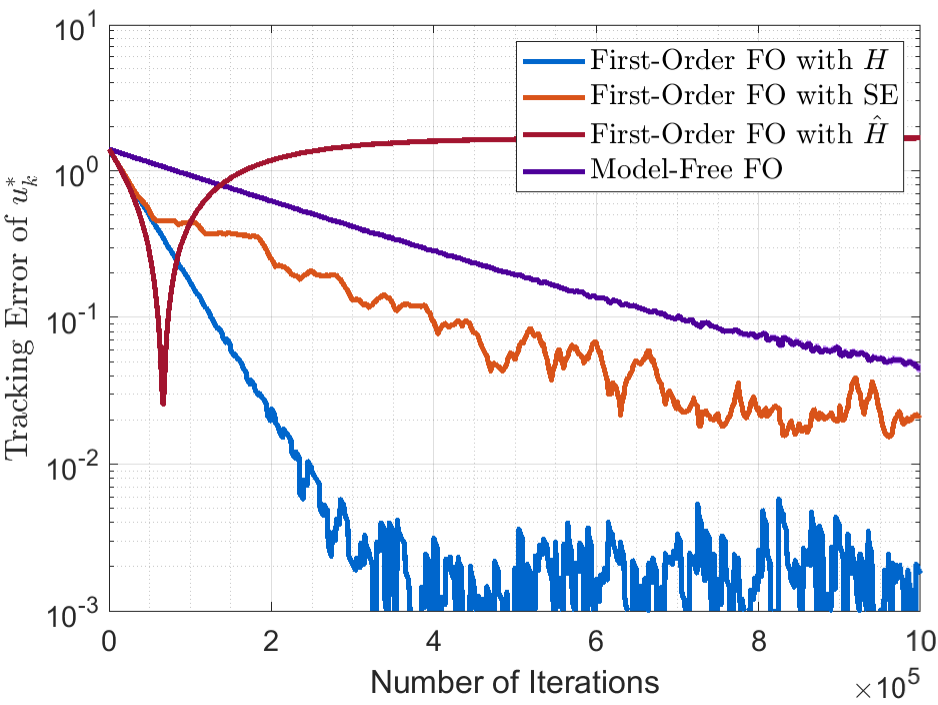}} \\
    \subfloat[Optimality gap $\tilde{\Phi}(u_k)-\Phi_k^*$.]{\includegraphics[width=0.71\columnwidth]{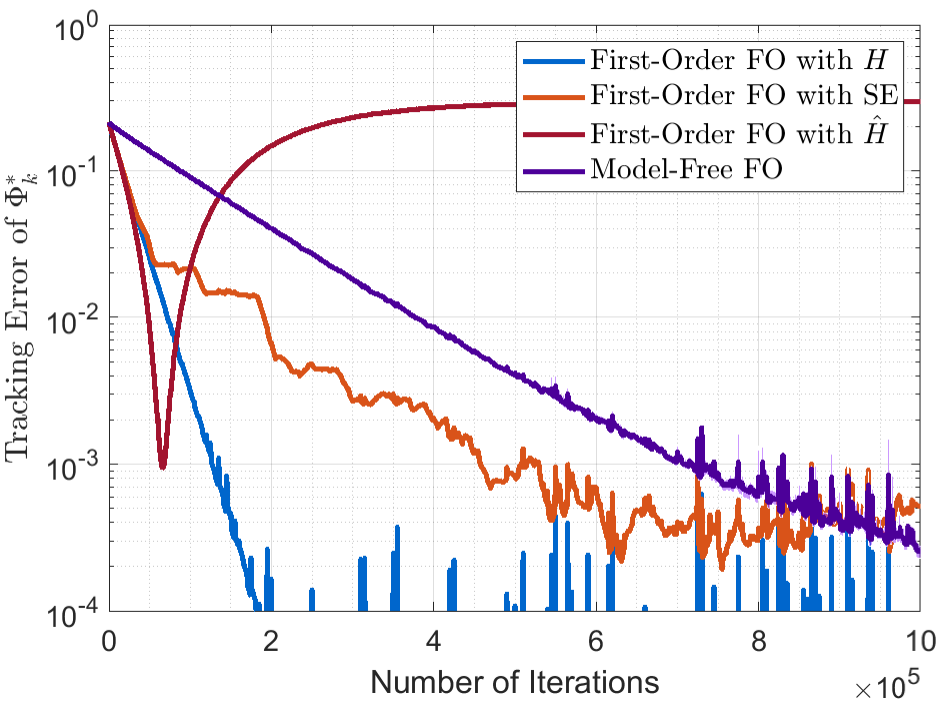}}
    \caption{Comparison of different methods in tracking the trajectory of time-varying optimal solutions.}
    \label{fig:trackingResult}
\end{figure}
\section{Conclusion}\label{sec:conclusion}
We considered the problem of designing a feedback controller to steer a physical system to an efficient operating condition, which is defined by a nonconvex optimization program. The proposed discrete-time controller is model-free, thanks to the design where gradient estimates are used to update control inputs. We constructed these estimates based on current and previous evaluations of the objective function. We analyzed the performance of the closed-loop interconnection of the nonlinear dynamical plant and the proposed controller. Furthermore, we established the dependence of the optimality of solutions on the problem dimension, the number of iterations, and the stability properties of the physical plant. Future directions include tuning step sizes in online implementations, tackling output constraints via dualization or output projection, quantifying the trade-offs between sensitivity estimation and model-free updates, especially when the steady-state map is non-differentiable at the optimal solution, as well as characterizing the tracking performance for time-varying problems. 
\appendix
\subsection{Proof of Lemma~\ref{lem:objErrorSubstitution}}\label{subsec:objErrorSubstitution}
Based on Assumption~\ref{assump:objective}, we have
\begin{align*}
    |e_{\Phi}(x_k,u_k)|^2 &= \big|\Phi(u_k,y_{k+1}) - \Phi(u_k,h(u_k,d))\big|^2 \\
        &\leq M_\Phi^2 \|y_{k+1}-h(u_k,d)\|^2 \\
        &\leq M_\Phi^2 M_g^2 \|x_{k+1}-x_{\text{ss}}(u_k,d)\|^2.
\end{align*}
It follows from \eqref{eq:Vbound} and \eqref{eq:Vdecrease} that
\begin{align}
    \|x_{k+1}&-x_{\text{ss}}(u_k,d)\|^2 \notag \\
            &\leq \frac{1}{\alpha_1} V(x_{k+1},u_k,d) = \frac{1}{\alpha_1} V(f(x_k,u_k,d),u_k,d) \notag \\
            &\leq \frac{1}{\alpha_1} \big(V(x_k,u_k,d) - \alpha_3 \|x_k-x_{\text{ss}}(u_k,d)\|^2\big) \notag \\
            &\leq \frac{1}{\alpha_1} \Big(1-\frac{\alpha_3}{\alpha_2}\Big) V(x_k,u_k,d). \label{eq:stateError}
\end{align}
Hence, based on the definition \eqref{eq:rate_mu} of $\mu$, inequality \eqref{eq:objSquareErrorSubstitution} holds.

\subsection{Proof of Lemma~\ref{lem:expectedLyapunovFuncDecrease}}\label{subsec:expectedLyapunovFuncDecrease}
Based on \eqref{eq:Vbound}, we have
\begin{align*}
    &V(x_k,u_k,d) \\
            &\leq \alpha_2 \|x_k-x_{\text{ss}}(u_k,d)\|^2 \\
            &= \alpha_2 \|x_k-x_{\text{ss}}(u_{k-1},d) + x_{\text{ss}}(u_{k-1},d)-x_{\text{ss}}(u_k,d)\|^2 \\
            &\stackrel{\text{(s.1)}}{\leq} 2\alpha_2\big(\|x_k-x_{\text{ss}}(u_{k-1},d)\|^2 + \|x_{\text{ss}}(u_{k-1},d)-x_{\text{ss}}(u_k,d)\|^2\big) \\
            &\stackrel{\text{(s.2)}}{\leq} 2\alpha_2\Big(\frac{1}{\alpha_1}\big(1-\frac{\alpha_3}{\alpha_2}\big)V(x_{k-1},u_{k-1},d) + M_x^2\|u_k-u_{k-1}\|^2\Big),
\end{align*}
where (s.1) follows from the inequality $\|a+b\|^2\leq 2(\|a\|^2+\|b\|^2), \forall a,b$, and (s.2) is built on \eqref{eq:stateError} and the Lipschitz continuity of $x_{\text{ss}}(u,d)$. By taking the expectation of both sides of the above inequality with respect to $v_{[k]}$ and using the definition \eqref{eq:rate_mu} of $\mu$, we obtain
\begin{align}\label{eq:VdecreasePerStepExpectMid}
    \E_{v_{[k]}}[V(x_k,u_k,d)] &\leq \mu \E_{v_{[k]}}[V(x_{k-1},u_{k-1},d)] \notag \\
            &\quad + 2\alpha_2 M_x^2 \E_{v_{[k]}}[\|u_k-u_{k-1}\|^2].
\end{align}
The bound on $\E_{v_{[k]}}[\|u_k-u_{k-1}\|^2]$ is given by
\begin{align*}
    \E_{v_{[k]}}[\|u_k&-u_{k-1}\|^2] = \E_{v_{[k]}}[\|w_k-w_{k-1}+\delta v_k-\delta v_{k-1}\|^2] \\
            &\stackrel{\text{(s.1)}}{\leq} \E_{v_{[k]}}[2\|w_k-w_{k-1}\|^2+2\delta^2\|v_k-v_{k-1}\|^2] \\
            &\stackrel{\text{(s.2)}}{\leq} 2\eta^2 \E_{v_{[k]}}[\|\tilde{\phi}_{k-1}\|^2] + 4\delta^2 p,
\end{align*}
where (s.1) follows by $\|a+b\|^2\leq 2(\|a\|^2+\|b\|^2), \forall a,b$; (s.2) relies on the independence between $v_k$ and $v_{k-1}$, and $\E[v]=0, \E[\|v\|^2]=p, \forall v\sim \mathcal{N}(0,I)$, see Lemma~\ref{lem:momentGaussianVector}. Hence, the inequality \eqref{eq:VdecreasePerStepExpect} holds.

\subsection{Proof of Lemma~\ref{lem:secondMomentDecrease}}\label{subsec:secondMomentDecrease}
First, we present a lemma proved in \cite{nesterov2017random} for the upper bounds on the moments of standard normal random vectors.
\begin{lemma}[\hspace{1sp}{\cite[Lemma~1]{nesterov2017random}}]\label{lem:momentGaussianVector}
    Let $v\in \mathbb{R}^{p}$ satisfy the standard multivariate normal distribution. Then,
    \begin{equation*}
        \E[\|v\|^t]\leq
        \begin{cases}
            p^{t/2}, & \text{if~} t\in [0,2], \\
            (p+t)^{t/2}, & \text{if~} t> 2.
        \end{cases}
    \end{equation*}
\end{lemma}
\noindent Two special cases include $\E[\|v\|^2]=p$ and $\E[\|v\|^4]\leq (p+4)^2, \forall v\sim \mathcal{N}(0,I_{p\times p})$.

Then, we focus on the bound for $\E_{v_{[k]}}[\|\tilde{\phi}_k\|^2]$. Based on \eqref{eq:updateDTestimate}, we have
\begin{align}
    &\E_{v_{[k]}}[\|\tilde{\phi}_k\|^2] = \frac{1}{\delta^2}\E_{v_{[k]}}\big[\|v_k(\Phi(u_k,y_{k+1})-\Phi(u_{k-1},y_k))\|^2\big] \notag \\
            &= \frac{1}{\delta^2}\E_{v_{[k]}}\big[\|v_k(\obj(u_k)-\obj(u_{k-1})) \notag \\
                    &\hspace{5em} + v_k(e_{\Phi}(x_k,u_k)-e_{\Phi}(x_{k-1},u_{k-1}))\|^2\big] \notag \\
            &\stackrel{\text{(s.1)}}{\leq} \underbrace{\frac{3}{\delta^2}\E_{v_{[k]}}\big[\|v_k(\obj(u_k)-\obj(u_{k-1}))\|^2\big]}_{\numcircled{1}} \notag \\
            &~ + \underbrace{\frac{3}{\delta^2}\E_{v_{[k]}}[\|v_k e_{\Phi}(x_k,u_k)\|^2]}_{\numcircled{2}} + \underbrace{\frac{3}{\delta^2}\E_{v_{[k]}}[\|v_k e_{\Phi}(x_{k-1},u_{k-1})\|^2]}_{\numcircled{3}}, \label{eq:secondMomentIneqBasic}
\end{align}
where (s.1) follows from the fact that $\forall a,b,c, \E[(a+b+c)^2]\leq 3\E[a^2+b^2+c^2] = 3\E[a^2]+3\E[b^2]+3\E[c^2]$.

First, we provide a bound on term \numcircled{1} in \eqref{eq:secondMomentIneqBasic}.
\begin{align*}
    \numcircled{1} &= \frac{3}{\delta^2}\E_{v_{[k]}}\big[\|v_k(\obj(w_k+\delta v_k) - \obj(w_{k-1}+\delta v_k))\|^2 \\
            & \qquad\quad + \obj(w_{k-1}+\delta v_k) - \obj(w_{k-1}+\delta v_{k-1}))\|^2 \big] \\
            &\stackrel{\text{(s.1)}}{=} \frac{6}{\delta^2}\E_{v_{[k]}}\big[\|v_k(\obj(w_k+\delta v_k) - \obj(w_{k-1}+\delta v_k))\|^2 \\
            & \qquad\quad + \|v_k(\obj(w_{k-1}+\delta v_k) - \obj(w_{k-1}+\delta v_{k-1}))\|^2 \big] \\
            &\stackrel{\text{(s.2)}}{\leq} \frac{6M^2}{\delta^2}\E_{v_{[k]}}\big[\|w_k-w_{k-1}\|^2 \|v_k\|^2 \\
            &\qquad\qquad\qquad~ + \delta^2\|v_k-v_{k-1}\|^2\|v_k\|^2 \big] \\
            &\stackrel{\text{(s.3)}}{\leq} \frac{6M^2}{\delta^2}\big(\eta^2\E_{v_{[k]}}[\|\tilde{\phi}_{k-1}\|^2]\E_{v_{[k]}}[\|v_k\|^2] \\
            &\qquad\qquad~ + \delta^2\E_{v_{[k]}}[2(\|v_k\|^2+\|v_{k-1}\|^2)\|v_k\|^2]\big) \\
            &\stackrel{\text{(s.4)}}{\leq} \frac{6}{\delta^2}M^2\eta^2 p\E_{v_{[k]}}[\|\tilde{\phi}_{k-1}\|^2] + 24M^2(p+4)^2,
\end{align*}
where $\text{(s.1)}$ relies on the fact that $\|a+b\|^2\leq 2(\|a\|^2+\|b\|^2), \forall a,b$; $\text{(s.2)}$ follows from the assumption that $\obj(u)$ is $M$-Lipschitz continuous; $\text{(s.3)}$ utilizes the independence between $v_k$ and $-\eta\tilde{\phi}_{k-1}$ and the inequality $\|a-b\|^{2}\leq 2(\|a\|^2+\|b\|^2), \forall a,b$; $\text{(s.4)}$ follows by Lemma~\ref{lem:momentGaussianVector}
and the independence between $v_k$ and $v_{k-1}$.

Then, we bound term \numcircled{2} in \eqref{eq:secondMomentIneqBasic}.
\begin{align*}
    \numcircled{2} &= \frac{3}{\delta^2} \E_{v_{[k]}}\left[\|v_k\|^2 \cdot |e_{\Phi}(x_k,u_k)|^2\right] \\
            &\stackrel{\text{(s.1)}}{\leq} \frac{3}{\delta^2} (\E_{v_{[k]}}[\|v_k\|^4])^{\frac{1}{2}} (\E_{v_{[k]}}[|e_{\Phi}(x_k,u_k)|^4])^{\frac{1}{2}} \\
            &\stackrel{\text{(s.2)}}{\leq} (p+4) \frac{3M_\Phi^2 M_g^2}{\alpha_1\delta^2} \Big(1-\frac{\alpha_3}{\alpha_2}\Big) \E_{v_{[k]}}[V(x_k,u_k,d)],
\end{align*}
where (s.1) relies on the Cauchy-Schwarz inequality, and its motivation is to split the product of two correlated random variables and then use the corresponding upper bounds; (s.2) follows by Lemma~\ref{lem:momentGaussianVector} 
and the bound on $|e_{\Phi}(x_k,u_k)|^4$ acquired by squaring both sides of \eqref{eq:objSquareErrorSubstitution}. 

Finally, we bound term \numcircled{3} in \eqref{eq:secondMomentIneqBasic}.
\begin{align*}
    \numcircled{3} &\stackrel{\text{(s.1)}}{=} \frac{3}{\delta^2} \E_{v_{[k]}}[\|v_k\|^2] \E_{v_{[k]}}[|e_{\Phi}(x_{k-1},u_{k-1})|^2] \\
            &\stackrel{\text{(s.2)}}{\leq} \frac{3p M_\Phi^2 M_g^2}{\alpha_1 \delta^2}\Big(1-\frac{\alpha_3}{\alpha_2}\Big) \E_{v_{[k]}}[V(x_{k-1},u_{k-1},d)],
\end{align*}
where (s.1) follows from the independence between $v_k$ and $e_{\Phi}(x_{k-1},u_{k-1})$; (s.2) follows by Lemma~\ref{lem:momentGaussianVector} 
and the bound on $|e_{\Phi}(x_k,u_k)|^2$ in \eqref{eq:objSquareErrorSubstitution}.

By combining the obtained upper bounds and using the definition \eqref{eq:rate_mu} of $\mu$, we arrive at \eqref{eq:secondMomentDecay}.

\subsection{A Bound on an Intermediate Term}\label{subsec:expectationMidTermBound}
We establish a bound on the term $\E_{v_{[k]}}[-\nabla \obj_{\delta}^{\top}(w_k)\tilde{\phi}_k]$, which is useful for characterizing the optimality of the steady state of the interconnection of \eqref{eq:updateOPT} and \eqref{eq:system}.
\begin{lemma}\label{lem:expectationMidTermBound}
    If Assumptions \ref{assump:system} and \ref{assump:objective} hold, with \eqref{eq:updateOPT}, we have
    \begin{equation}\label{eq:expectationMidTermFinal}
    \begin{split}
        \E_{v_{[k]}}[-\nabla &\obj_{\delta}^{\top}(w_k)\tilde{\phi}_k] \leq -\frac{1}{2} \E_{v_{[k]}}[\|\nabla \obj_{\delta}(w_k)\|^2] \\
            &+ \frac{p M_\Phi^2 M_g^2}{2\delta^2 \alpha_1}\Big(1-\frac{\alpha_3}{\alpha_2}\Big)\E_{v_{[k]}}[V(x_k,u_k,d)].
    \end{split}
    \end{equation}
\end{lemma}
\begin{proof}
    Note that
    \begin{align}\label{eq:expectationMidTerm}
        \E_{v_{[k]}}&[-\nabla \obj_{\delta}^{\top}(w_k)\tilde{\phi}_k] \notag \\
            &\stackrel{\text{(s.1)}}{=} \E_{v_{[k-1]}}\big[\E_{v_k}[-\nabla \obj_{\delta}^{\top}(w_k)\tilde{\phi}_k|v_{[k-1]}]\big] \notag \\
            &\stackrel{\text{(s.2)}}{=} \E_{v_{[k-1]}}\big[\nabla \obj_{\delta}^{\top}(w_k)\E_{v_k}[-\tilde{\phi}_k|v_{[k-1]}]\big],
    \end{align}
    where $\text{(s.1)}$ follows from the tower rule; $\text{(s.2)}$ holds since $\obj_{\delta}(w_k)$ is measurable with respect to $v_{[k-1]}$.
    Next, we focus on $\E_{v_k}[-\tilde{\phi}_k|v_{[k-1]}]$.
    \begin{align}\label{eq:expectationGradEstimate}
        \E_{v_k}&[-\tilde{\phi}_k|v_{[k-1]}] \notag \\
         &= \E_{v_k}\Big[-\frac{v_k}{\delta}\big(\obj(u_k)+e_{\Phi}(x_k,u_k)\big)\big|v_{[k-1]}\Big] \notag \\
                &\qquad + \E_{v_k}\Big[\frac{v_k}{\delta}\big(\obj(u_{k-1})+e_{\Phi}(x_{k-1},u_{k-1})\big)\big|v_{[k-1]}\Big] \notag \\
                &\stackrel{\text{(s.1)}}{=} \E_{v_k}\Big[-\frac{v_k}{\delta}\obj(w_k+\delta v_k)\big|v_{[k-1]}\Big]  \notag \\
                &\qquad + \E_{v_k}\Big[-\frac{v_k}{\delta}e_{\Phi}(x_k,u_k)\big|v_{[k-1]}\Big] \notag \\
                &\stackrel{\text{(s.2)}}{=} -\nabla \obj_{\delta}(w_k) + \E_{v_k}\Big[-\frac{v_k}{\delta}e_{\Phi}(x_k,u_k)\big|v_{[k-1]}\Big],
    \end{align}
    where (s.1) follows by the independence of $\obj(u_{k-1})$ and $e_{\Phi}(x_{k-1},u_{k-1})$ on $v_k$ and $\E_{v_k}[v_k]=0$; (s.2) utilizes \cite[Eq.~(21)]{nesterov2017random} (see also \cite[Lemma~3.1]{zhang2020improving}). By plugging \eqref{eq:expectationGradEstimate} into \eqref{eq:expectationMidTerm} and utilizing the inequality $\forall z_1,z_2,\E[z_1^{\top}z_2]\leq (\E[\|z_1\|^2]\E[\|z_2\|^2])^{\frac{1}{2}}\leq \frac{1}{2}(\E[\|z_1\|^2]+\E[\|z_2\|^2])$, we obtain
    \begin{align}\label{eq:expectationMidTermII}
        \E_{v_{[k]}}&[-\nabla \obj_{\delta}^{\top}(w_k)\tilde{\phi}_k] \notag \\
                &= -\E_{v_{[k]}}[\|\nabla \obj_{\delta}(w_k)\|^2] \notag \\
                &\quad + \E_{v_{[k-1]}}\Big[\nabla \obj_{\delta}^{\top}(w_k)\E_{v_k}\big[-\frac{v_k}{\delta}e_{\Phi}(x_k,u_k)\big|v_{[k-1]}\big]\Big] \notag \\
                &\leq -\E_{v_{[k]}}[\|\nabla \obj_{\delta}(w_k)\|^2] + \frac{1}{2}\E_{v_{[k-1]}}[\|\nabla \obj_{\delta}(w_k)\|^2] \notag \\
                &\quad + \frac{1}{2}\E_{v_{[k-1]}}\Big[\big\|\E_{v_k}\big[\frac{v_k}{\delta}e_{\Phi}(x_k,u_k)\big|v_{[k-1]}\big]\big\|^2\Big].
    \end{align}
    For the last term of the right-hand side of \eqref{eq:expectationMidTermII}, we have
    \begin{align}\label{eq:expectationMidTermIII}
        &\E_{v_{[k-1]}}\Big[\big\|\E_{v_k}\big[\frac{v_k}{\delta}e_{\Phi}(x_k,u_k)\big|v_{[k-1]}\big]\big\|^2\Big] \notag \\
                &~\stackrel{\text{(s.1)}}{=} \E_{v_{[k-1]}}\Big[\sum_{j=1}^{p} \frac{1}{\delta^2}\E_{v_k}^2[e_{\Phi}(x_k,u_k) v_k(j)|v_{[k-1]}]\Big] \notag \\
                &~\stackrel{\text{(s.2)}}{\leq} \frac{1}{\delta^2} \E_{v_{[k-1]}}\Big[\sum_{j=1}^{p}\E_{v_k}[e_{\Phi}^2(x_k,u_k)|v_{[k-1]}]\E_{v_k}[v_k^2(j)|v_{[k-1]}]\Big] \notag \\
                &~\stackrel{\text{(s.3)}}{=} \frac{p}{\delta^2} \E_{v_{[k-1]}}\big[\E_{v_k}[e_{\Phi}^2(x_k,u_k)|v_{[k-1]}]\big] \notag \\
                &~\stackrel{\text{(s.4)}}{\leq} \frac{p M_\Phi^2 M_g^2}{\delta^2 \alpha_1}\Big(1-\frac{\alpha_3}{\alpha_2}\Big)\E_{v_{[k]}}[V(x_k,u_k,d)],
    \end{align}
    where in (s.1), $v_k(j)$ denotes the $j$-th component of $v_k$; (s.2) follows from the Cauchy-Schwarz inequality; (s.3) uses $\E_{v_k}[v_k^2(j)]=1, \forall v_k(j)\sim \mathcal{N}(0,1)$; (s.4) relies on \eqref{eq:objSquareErrorSubstitution} and the tower rule. By combining \eqref{eq:expectationMidTermIII} with \eqref{eq:expectationMidTermII}, we obtain \eqref{eq:expectationMidTermFinal}.
\end{proof}

\subsection{An Auxiliary Lemma}\label{subsec:auxiliaryLemma}
The following auxiliary lemma gives the bounds on the partial sums of two coupled sequences of nonnegative numbers, and it is inspired by \cite[Proof of Theorem~1]{belgioioso2021sampled}. Let $(g_k)_{k\in \mathbb{N}}$ and $(h_k)_{k\in \mathbb{N}}$ be two nonnegative sequences that satisfy
\begin{subequations}\label{eq:coupledSequence}
    \begin{align}
        g_k &\leq a_1g_{k-1} + b_1h_{k-1} + \tau_1, \label{eq:coupledSequence-g} \\
        h_k &\leq a_2g_{k-1} + b_2h_{k-1} + \tau_2. \label{eq:coupledSequence-h}
    \end{align}
\end{subequations}
Suppose that the matrix of coefficients $C$ satisfies
\begin{equation}\label{eq:maxEigenvalue}
    C = \begin{bmatrix} a_1 & b_1 \\ a_2 & b_2 \end{bmatrix}, \quad \|C\| = \sigma_{\text{max}}(C) \triangleq \sigma < 1.
\end{equation}
Let the $T$-th partial sums of $(g_k)_{k\in \mathbb{N}}$ and $(h_k)_{k\in \mathbb{N}}$ be $G_T = \sum_{k=0}^{T} g_k$ and $H_T = \sum_{k=0}^{T} h_k$, 
respectively. The bounds on $G_T$ and $H_T$ are provided in the following lemma.

\begin{lemma}\label{lem:boundPartialSumSequence}
    With nonnegative sequences \eqref{eq:coupledSequence} and the matrix of coefficients \eqref{eq:maxEigenvalue}, we have
    \begin{equation}\label{eq:boundPartialSumSequence}
        \max\{G_T,H_T\} \leq \Big(\sigma^{T} + \frac{1}{1-\sigma}\Big)(g_0+h_0) + \frac{T}{1-\sigma}(\tau_1+\tau_2).
    \end{equation}
\end{lemma}
\begin{proof}
    By summing over both sides of \eqref{eq:coupledSequence} when $k=T,T-1,\ldots,1$, we have
    \begin{align*}
        G_T-g_0 &\leq a_1G_{T-1} + b_1H_{T-1} + T\tau_1, \\
        H_T-h_0 &\leq a_2G_{T-1} + b_2H_{T-1} + T\tau_2,
    \end{align*}
    which can be compactly written as
    \begin{equation*}
        \begin{bmatrix}
            G_T \\ H_T
        \end{bmatrix}
        \leq C
        \begin{bmatrix}
            G_{T-1} \\ H_{T-1}
        \end{bmatrix}
        +
        \begin{bmatrix}
            T\tau_1 + g_0 \\ T\tau_2 + h_0
        \end{bmatrix}.
    \end{equation*}
    By recursively using the above inequality, we have
    \begin{equation*}
        \begin{bmatrix}
            G_T \\ H_T
        \end{bmatrix}
        \leq C^T
        \begin{bmatrix}
            g_{0} \\ h_{0}
        \end{bmatrix}
        + \sum_{k=1}^{T} C^{T-k}
        \begin{bmatrix}
            k\tau_1 + g_0 \\ k\tau_2 + h_0
        \end{bmatrix}.
    \end{equation*}
    Note that $G_T$ and $H_T$ are nonnegative for any $T\in \mathbb{N}$. It follows that
    \begin{align*}
        \max\{G_T&,H_T\} \leq \bigg\|\begin{bmatrix} G_T \\ H_T \end{bmatrix}\bigg\| \\
                &\stackrel{\text{(s.1)}}{\leq} \sigma^{T} \bigg\|\begin{bmatrix} g_0 \\ h_0 \end{bmatrix}\bigg\| + \sum_{k=1}^{T} \sigma^{T-k} \bigg\|\begin{bmatrix} T\tau_1 + g_0 \\ T\tau_2 + h_0 \end{bmatrix}\bigg\| \\
                &\stackrel{\text{(s.2)}}{\leq} \sigma^{T}(g_0+h_0) + \frac{1}{1-\sigma}(T(\tau_1+\tau_2)+g_0+h_0),
    \end{align*}
    where (s.1) relies on the sub-multiplicativity of the norms and $\|C\|=\sigma$; (s.2) follows from the fact that $\|[a~b]^{\top}\|\leq a+b$ for any nonnegative numbers $a$ and $b$. Hence, \eqref{eq:boundPartialSumSequence} holds.
\end{proof}

\subsection{Proof of Theorem~\ref{thm:optimality}}\label{subsec:optimality}
It follows from Assumption~\ref{assump:objective} that the Gaussian smooth approximation $\obj_{\delta}(w)$ is $M\sqrt{p}/\delta(\triangleq L)$-smooth. Hence,
\begin{align}\label{eq:GaussianSmoothIneq}
    \obj_{\delta}&(w_{k+1}) \notag \\
        &\leq \obj_{\delta}(w_k) + \nabla \obj_{\delta}^{\top}(w_k)(w_{k+1}-w_{k}) + \frac{L}{2} \|w_{k+1}-w_{k}\|^2 \notag \\
        &= \obj_{\delta}(w_k) - \eta\nabla \obj_{\delta}^{\top}(w_k)\tilde{\phi}_k + \frac{L\eta^2}{2}\|\tilde{\phi}_k\|^2.
\end{align}
By taking expectations of both sides of \eqref{eq:GaussianSmoothIneq} with respect to $v_{[k]}$ and referring to Lemma~\ref{lem:expectationMidTermBound} in Appendix~\ref{subsec:expectationMidTermBound}, we have
\begin{equation*}
    \begin{split}
        \E_{v_{[k]}}[\obj_{\delta}(w_{k+1})] &\leq \E_{v_{[k]}}[\obj_{\delta}(w_{k})] - \frac{\eta}{2} \E_{v_{[k]}}[\|\nabla \obj_{\delta}(w_k)\|^2] \\
            &\quad + \frac{\eta p M_\Phi^2 M_g^2}{2\delta^2 \alpha_1}\Big(1-\frac{\alpha_3}{\alpha_2}\Big)\E_{v_{[k]}}[V(x_k,u_k,d)] \\
            &\quad + \frac{L\eta^2}{2}\E_{v_{[k]}}[\|\tilde{\phi}_k\|^2].
    \end{split}
\end{equation*}
By rearranging terms and telescoping sums, we obtain
\begin{align}\label{eq:ncvxBoundMid}
    &\sum_{k=0}^{T-1} \E_{v_{[T-1]}}[\|\nabla \obj_{\delta}(w_k)\|^2] \notag \\
            &~\leq \frac{2}{\eta} \big(\E_{v_{[T-1]}}[\obj_{\delta}(w_0)] - \E_{v_{[T-1]}}[\obj_{\delta}(w_{T})]\big) \notag \\
            &~\qquad + L\eta \sum_{k=0}^{T-1}\E_{v_{[T-1]}}[\|\tilde{\phi}_k\|^2] \notag \\
            &~\qquad + \frac{p M_\Phi^2 M_g^2}{\delta^2 \alpha_1}\Big(1-\frac{\alpha_3}{\alpha_2}\Big)\sum_{k=0}^{T-1}\E_{v_{[T-1]}}[V(x_k,u_k,d)].
\end{align}

We now focus on the last two terms of \eqref{eq:ncvxBoundMid}. By incorporating \eqref{eq:VdecreasePerStepExpect} into the right-hand side of \eqref{eq:secondMomentDecay}, we have
\begin{equation*} 
    \begin{bmatrix}
        \E_{v_{[k]}}[\|\tilde{\phi}_k\|^2] \\ \E_{v_{[k]}}[V(x_k,u_k,d)]
    \end{bmatrix} \preceq
    C
    \begin{bmatrix}
        \E_{v_{[k]}}[\|\tilde{\phi}_{k-1}\|^2] \\ \E_{v_{[k]}}[V(x_{k-1},u_{k-1},d)]
    \end{bmatrix} +
    \begin{bmatrix}
        \tau_1 \\ \tau_2
    \end{bmatrix},
\end{equation*}
where $C = \begin{bmatrix} c_{11} & c_{12} \\ c_{21} & c_{22} \end{bmatrix}$ and
\begin{align}\label{eq:parametersDecay}
    c_{11} &= \frac{6\eta^2}{\delta^2}\big(M^2p + M_\Phi^2M_g^2M_x^2(p+4)\mu\big), \notag \\
    c_{12} &= \frac{3M_\Phi^2M_g^2}{\alpha_1\delta^2}\big(1-\frac{\alpha_3}{\alpha_2}\big)\big(p+(p+4)\mu\big), \notag \\
    c_{21} &= 4\alpha_2\eta^2M_x^2, \notag \\
    c_{22} &= \mu, \notag \\
    \tau_1 &= 24M^2(p+4)^2 + 12M_\Phi^2M_g^2M_x^2p(p+4)\mu, \notag \\
    \tau_2 &= 8\alpha_2p\delta^2M_x^2.
\end{align}
The above inequality can be reformulated as
\begin{align*}
    &
    \begin{bmatrix}
        \E_{v_{[k]}}[\|\tilde{\phi}_k\|^2] \\ \E_{v_{[k]}}\Big[\sqrt{\frac{c_{12}}{c_{21}}} V(x_k,u_k,d)\Big]
    \end{bmatrix} \preceq \\
    &\qquad \qquad C'
    \begin{bmatrix}
        \E_{v_{[k]}}[\|\tilde{\phi}_{k-1}\|^2] \\ \E_{v_{[k]}}\Big[\sqrt{\frac{c_{12}}{c_{21}}}V(x_{k-1},u_{k-1},d)\Big]
    \end{bmatrix}
    +
    \begin{bmatrix}
        \tau_1 \\ \sqrt{\frac{c_{12}}{c_{21}}} \tau_2
    \end{bmatrix},
\end{align*}
where the elements of $C'$ are given by
\begin{equation}\label{eq:parametersDecayTransformed}
    c'_{11} = c_{11}, \quad c'_{12} = c'_{21} = \sqrt{c_{12}c_{21}}, \quad c'_{22} = c_{22}.
\end{equation}
Since $C'$ is a symmetric matrix, its spectral norm is $\|C'\| = \rho(C') \triangleq \rho$. Additionally, $C'$ is a positive matrix. It follows from the Perron-Frobenius theorem\cite{horn2012matrix} that $\rho$ equals the Perron eigenvalue of $C'$. Hence,
\begin{align*}
   \rho &= \frac{c'_{11}+c'_{22}}{2} + \sqrt{\Big(\frac{c'_{11}-c'_{22}}{2}\Big)^2 + c'_{12}c'_{21}} \\ 
        &\leq \frac{c'_{11}+c'_{22}}{2} + \Big|\frac{c'_{11}-c'_{22}}{2}\Big| + \sqrt{c'_{12}c'_{21}} \\
        &= \max\{c'_{11},c'_{22}\} + \sqrt{c'_{12}c'_{21}}.
\end{align*}
To guarantee that $\rho < 1$, we need to set $\delta$ and $\eta$ such that
\begin{equation}\label{eq:rhoRequirement}
    c'_{11} + \sqrt{c'_{12}c'_{21}} < 1, \qquad c'_{22} + \sqrt{c'_{12}c'_{21}} < 1.
\end{equation}
Note that $\E_{v_{[k]}}[\|\tilde{\phi}_k\|^2]$ and $\E_{v_{[k]}}[\sqrt{c_{12}/c_{21}}V(x_k,u_k,d)]$ are non-negative for any $k\in \mathbb{N}$. Hence, from Lemma~\ref{lem:boundPartialSumSequence} in Appendix~\ref{subsec:auxiliaryLemma}, we have
\begin{align}\label{eq:boundPartialSumOverall}
        \max&\Big\{\sum_{k=0}^{T-1}\E_{v_{[T-1]}}[\|\tilde{\phi}_k\|^2], \sum_{k=0}^{T-1}\E_{v_{[T-1]}}\big[{\textstyle \sqrt{\frac{c_{12}}{c_{21}}}}V(x_k,u_k,d)\big]\Big\} \notag \\
            &\leq \Big(\rho^{T-1}+\frac{1}{1-\rho}\Big)B_0 + \frac{T-1}{1-\rho}\big(\tau_1+{\textstyle \sqrt{\frac{c_{12}}{c_{21}}}}\tau_2\big),
\end{align}
where $B_0 = \E[\|\tilde{\phi}(w_0)\|^2]+\E[\sqrt{c_{12}/c_{21}}V(x_0,u_0,d)]$. By combining \eqref{eq:boundPartialSumOverall} with \eqref{eq:ncvxBoundMid} and dividing both sides of the inequality by $T$, we have
\begin{equation}\label{eq:sumGradMomentBound}
    \begin{split}
        \frac{1}{T}&\sum_{k=0}^{T-1} \E_{v_{[T-1]}}[\|\nabla \obj_{\delta}(w_k)\|^2] \leq \frac{2}{\eta T}\big(\E[\obj_{\delta}(w_0)]-\obj_{\delta}^{*}\big) \\
            & + \frac{1}{T}\Big(\frac{\eta M\sqrt{p}}{\delta} + \frac{p M_\Phi^2 M_g^2}{\delta^2 \alpha_1}{\textstyle \sqrt{\frac{c_{21}}{c_{12}}}}\Big(1-\frac{\alpha_3}{\alpha_2}\Big)\Big) \\ 
            &\quad \cdot \left[\Big(\rho^{T-1}+\frac{1}{1-\rho}\Big)B_0 + \frac{T-1}{1-\rho}\big(\tau_1+{\textstyle \sqrt{\frac{c_{12}}{c_{21}}}}\tau_2\big)\right],
    \end{split}
\end{equation}
where we use the fact that $\E[\obj_{\delta}(w_{T})]\geq \obj_{\delta}^{*}\triangleq \inf_{u\in \mathbb{R}^p} \obj_{\delta}(u)$. Note that $\obj_{\delta}^{*}$ is finite, since $\inf_{u\in \mathbb{R}^p} \obj(u)$ is finite (see Assumption~\ref{assump:objective}) and $\obj_{\delta}(u)$ is close to $\obj(u)$ (see Lemma~\ref{lem:gaussianSmoothApprox}).

To ensure that $|\obj_{\delta}(u) - \obj(u)|\leq \epsilon$, we set $\delta = \frac{\epsilon}{M\sqrt{p}}$. With direct calculations, we obtain that the following terms are of the orders of
\begin{align*}
    &\text{i.~} {\textstyle \sqrt{\frac{c_{12}}{c_{21}}}} = \bigO{\frac{p}{\eta \epsilon} \Big(\frac{1}{\alpha_1\alpha_2} \big(1-\frac{\alpha_3}{\alpha_2}\big)\Big)^{\frac{1}{2}}}, \\
    &\text{ii.~} \frac{\eta M\sqrt{p}}{\delta} + \frac{p M_\Phi^2 M_g^2}{\delta^2 \alpha_1}{\textstyle \sqrt{\frac{c_{21}}{c_{12}}}}\Big(1-\frac{\alpha_3}{\alpha_2}\Big) = \bigO{\frac{p\eta}{\epsilon}}, \\
    &\text{iii.~} \Big(\rho^{T-1}+\frac{1}{1-\rho}\Big)B_0 + \frac{T-1}{1-\rho}\big(\tau_1+{\textstyle \sqrt{\frac{c_{12}}{c_{21}}}}\tau_2\big) \\
    &\qquad\qquad  = \bigO{\frac{p}{1-\rho}\bigg(Tp + \frac{T\epsilon+\frac{1}{\epsilon \alpha_2}}{\eta} \sqrt{\mu} \bigg)}.
\end{align*}
Let $\eta = \frac{\kappa \sqrt{\epsilon}}{p^{\frac{3}{2}}\sqrt{T}}$ such that $\frac{1}{\eta T}$ and $\frac{p^3\eta}{\epsilon}$ are of the same order. Then, based on the definition \eqref{eq:rate_mu} of $\mu$, the upper bound on the right-hand side of \eqref{eq:sumGradMomentBound} is of the order of
\begin{equation*}
    \bigO{\frac{p^\frac{3}{2}}{\sqrt{\epsilon T}(1-\rho)}} + \bigO{\frac{p^2}{1-\rho} \cdot \Big(1+\frac{1}{T\epsilon^2 \alpha_2}\Big) \cdot \sqrt{\mu}}.
\end{equation*}
The parameter $\kappa$ is set to meet the requirement \eqref{eq:rhoRequirement}, i.e.,
\begin{equation}\label{eq:kappa2Bound}
        \zeta_1 \kappa^2 + \zeta_2 \kappa < 1, \qquad
        \zeta_3 + \zeta_2 \kappa < 1,
\end{equation}
where
\begin{align*}
    \zeta_1 &= \frac{6M^2}{p^2T\epsilon}\big(M^2p + M_\Phi^2M_g^2M_x^2(p+4)\mu\big), \\
    \zeta_2 &= \Big({\textstyle \frac{6M^2M_x^2M_\Phi^2M_g^2\mu}{p^2T\epsilon}\big(p+(p+4)\mu\big)}\Big)^{\frac{1}{2}}, \\
    \zeta_3 &= \mu.
\end{align*}
The feasible range is denoted by $(0,\kappa^*)$. It is nonempty provided that $\zeta_{3} <1$ (see Assumption~\ref{assump:systemParam}). Based on \eqref{eq:kappa2Bound}, the parametric conditions, and the definition \eqref{eq:rate_mu} of $\mu$, we have
\begin{align*}
    \kappa^* &= \min\left\{\frac{-\zeta_2+\sqrt{\zeta_2^2+4\zeta_1}}{2\zeta_1}, \frac{1-\zeta_3}{\zeta_2} \right\} \\
         &= \bigO{\min\Big\{1-\sqrt{\mu},\frac{1-\mu}{\sqrt{\mu}}\Big\}\cdot \sqrt{pT\epsilon}}, \\
    \rho &= \max\left\{\zeta_1\kappa^2,\zeta_3\right\}+\zeta_2\kappa \\
         &= \bigO{\max\Big(\frac{1}{Tp\epsilon},\mu\Big) + \Big(\frac{\mu}{Tp\epsilon}\Big)^{\frac{1}{2}}}.
\end{align*}

\subsection{Proof of Corollary~\ref{cor:suboptimalityCor}}\label{subsec:suboptimalityCor}
The bound on $\|\nabla \obj_{\delta}(u_k)\|^2$ is given by
\begin{align}\label{eq:boundTransientSecondMoment}
    \|\nabla \obj_{\delta}(u_k)\|^2 &\stackrel{\text{(s.1)}}{\leq} 2\big(\|\nabla \obj_{\delta}(w_k)\|^2 + \|\nabla \obj_{\delta}(w_k) - \nabla \obj_{\delta}(u_k)\|^2 \big) \notag \\
    &\stackrel{\text{(s.2)}}{\leq} 2\Big(\|\nabla \obj_{\delta}(w_k)\|^2 + \big(\frac{M\sqrt{p}}{\delta}\delta \|v_k\|\big)^2\Big) \notag \\
    &\leq 2(\|\nabla \obj_{\delta}(w_k)\|^2 + M^2p \|v_k\|^2),
\end{align}
where (s.1) follows by the inequality $\forall a,b,\|b\|^2=\|a-(a-b)\|^2\leq 2(\|a\|^2+\|a-b\|^2)$; (s.2) uses the property that $\nabla \obj_{\delta}(u)$ is $M\sqrt{p}/\delta$-Lipschitz. By taking expectations of both sides of \eqref{eq:boundTransientSecondMoment} with respect to $v_{[T-1]}$, we obtain
\begin{align*}
    \E_{v_{[T-1]}}&[\|\nabla \obj_{\delta}(u_k)\|^2] \\
        &\leq 2\big(\E_{v_{[T-1]}}[\|\nabla \obj_{\delta}(w_k)\|^2] + M^2p\E_{v_{[T-1]}}[\|v_k\|^2]\big) \\
        &= 2\E_{v_{[T-1]}}[\|\nabla \obj_{\delta}(w_k)\|^2] + 2M^2p^2,
\end{align*}
where the last equality is due to $\E[\|v_k\|^2] = p, \forall v_k\sim \mathcal{N}(0,I_{p\times p})$. It follows that \eqref{eq:boundExpectedTransientSecondMoment} holds.

\subsection{Proof of Theorem~\ref{thm:stability}}\label{subsec:stability}
Note that
\begin{align}\label{eq:boundStateDist}
    \frac{1}{T} &\sum_{k=1}^{T-1} \E_{v_{[T-1]}}[\|x_{k+1} - x_{\text{ss}}(u_k,d)\|^2] \notag \\
        &\stackrel{\text{(s.1)}}{\leq} \frac{\mu}{2\alpha_2 T} \sum_{k=1}^{T-1} \E_{v_{[T-1]}}[V(x_k,u_k,d)] \notag \\
        &\stackrel{\text{(s.2)}}{\leq} \frac{\mu}{2\alpha_2 T} \cdot {\textstyle \sqrt{\frac{c_{21}}{c_{12}}}} \cdot \left[\Big(\rho^{T-1}+\frac{1}{1-\rho}\Big)B_0\right. \notag \\
        &\qquad + \left. \frac{T-1}{1-\rho}\big(\tau_1+{\textstyle \sqrt{\frac{c_{12}}{c_{21}}}}\tau_2\big)\right],
\end{align}
where (s.1) follows from \eqref{eq:stateError} in Appendix~\ref{subsec:objErrorSubstitution} and the definition \eqref{eq:rate_mu} of $\mu$, and (s.2) uses \eqref{eq:boundPartialSumOverall} in Appendix~\ref{subsec:optimality}. Based on the mentioned parametric conditions and the orders of terms derived in Appendix~\ref{subsec:optimality}, we obtain the order on the right-hand side of \eqref{eq:boundOrderStateDist}. 

\subsection{Proof of Theorem~\ref{thm:constrainedExtension}}\label{subsec:constrainedExtension}
Based on Lemma~\ref{lem:uniformSmoothApprox}, $\obj_{\delta}(w)$ is $L(=Mp/\delta)$-smooth. Hence,
\begin{align}\label{eq:GaussianSmoothIneqCons}
    &\obj_{\delta}(w_{k+1}) \notag \\
        &\leq \obj_{\delta}(w_k) + \nabla\obj_{\delta}^{\top}(w_k)(w_{k+1}-w_{k}) + \frac{L}{2} \|w_{k+1}-w_{k}\|^2 \notag \\
        &= \obj_{\delta}(w_k) + \eta \nabla\obj_{\delta}^{\top}(w_k)(s_k\!-\!w_k) + \frac{L\eta^2}{2} \|s_k\!-\!w_k\|^2,
\end{align}
where the equality follows from the update \eqref{eq:updateOPTCons}. Let
\begin{equation*}
    \hat{s}_k = \argmin_{s\in \mathcal{U}} \langle s, \nabla \obj_{\delta}(w_k) \rangle
\end{equation*}
be an auxiliary variable. It follows from the calculation of $s_k$ \eqref{eq:LMCons} and the definition of $\mathcal{G}(w_k)$ \eqref{eq:FWgap} that
\begin{subequations}
\begin{align}
    \langle s_k,\tilde{\phi}_k \rangle &\leq \langle \hat{s}_k,\tilde{\phi}_k \rangle, \label{eq:productBound} \\
    -\mathcal{G}(w_k) &= \min_{s\in \mathcal{U}} \langle s-w_k,\nabla \obj_{\delta}(w_k) \rangle \notag \\
        &= \langle \hat{s}_k-w_k,\nabla \obj_{\delta}(w_k) \rangle. \label{eq:FWgapReformulated}
\end{align}
\end{subequations}
For the middle term of the upper bound in \eqref{eq:GaussianSmoothIneqCons}, we have
\begin{align}\label{eq:midtermBoundCons}
    \eta \nabla&\obj_{\delta}^{\top}(w_k)(s_k-w_k) \notag \\
        &\stackrel{\text{(s.1)}}{=} \eta \tilde{\phi}_k^{\top}(s_k-w_k) + \eta(\nabla\obj_{\delta}(w_k)-\tilde{\phi}_k)^{\top}(s_k-w_k) \notag \\
        &\stackrel{\text{(s.2)}}{\leq} \eta \tilde{\phi}_k^{\top}(\hat{s}_k-w_k) + \eta(\nabla\obj_{\delta}(w_k)-\tilde{\phi}_k)^{\top}(s_k-w_k) \notag \\
        &\stackrel{\text{(s.3)}}{=} \eta \nabla\obj_{\delta}^{\top}(w_k)(\hat{s}_k-w_k) + \eta(\tilde{\phi}_k-\nabla\obj_{\delta}(w_k))^{\top}(s_k-\hat{s}_k) \notag \\
        &\stackrel{\text{(s.4)}}{=} -\eta \mathcal{G}(w_k) + \eta\|\tilde{\phi}_k-\nabla\obj_{\delta}(w_k)\|\|s_k-\hat{s}_k\|,
\end{align}
where (s.1) is obtained by adding and subtracting $\eta\tilde{\phi}_k^{\top}(s_k-w_k)$; (s.2) follows from \eqref{eq:productBound}; (s.3) follows by adding and subtracting $\nabla\obj_{\delta}^{\top}(w_k)(\hat{s}_k-w_k)$; (s.4) uses \eqref{eq:FWgapReformulated} and the Cauchy-Schwarz inequality. We incorporate \eqref{eq:midtermBoundCons} into \eqref{eq:GaussianSmoothIneqCons}, take expectations of both sides of the inequality with respect to $v_{[k]}$, and obtain
\begin{align*}
    &\E_{v_{[k]}}[\obj_{\delta}(w_{k+1})] \\
        &\leq \E_{v_{[k]}}[\obj_{\delta}(w_k)] - \eta\E_{v_{[k]}}[\mathcal{G}(w_k)] + \frac{L\eta^2}{2}\E_{v_{[k]}}[\|s_k-w_k\|^2] \\
        &\quad + \eta\E_{v_{[k]}}[\|\tilde{\phi}_k-\nabla\obj_{\delta}(w_k)\|\|s_k-\hat{s}_k\|] \\
        &\leq \E_{v_{[k]}}[\obj_{\delta}(w_k)] - \eta\E_{v_{[k]}}[\mathcal{G}(w_k)] + \frac{L\eta^2D^2}{2} \\
        &\quad + \eta D\E_{v_{[k]}}[\|\tilde{\phi}_k-\nabla\obj_{\delta}(w_k)\|] ,
\end{align*}
where the last inequality holds since $\mathcal{U}$ is bounded with diameter $D$, and the update \eqref{eq:updateDTCons} and the choice of $\eta$ guarantee that $s_k,\hat{s}_k,w_k\in \mathcal{U}$. By rearranging terms, telescoping sums, and dividing both sides of the inequality by $T$, we have
\begin{align}\label{eq:ncvxConsBoundMid}
    \frac{1}{T}\sum_{k=1}^{T}&\E_{v_{[T]}}[\mathcal{G}(w_k)] \notag \\
        &\leq \frac{1}{\eta T}\big(\E_{v_{[T]}}[\obj_{\delta}(w_1)]-\E_{v_{[T]}}[\obj_{\delta}(w_{T+1})]\big) + \frac{L\eta D^2}{2} \notag \\
        &\quad + \frac{D}{T}\sum_{k=1}^{T}\E_{v_{[T]}}[\|\tilde{\phi}_k-\nabla\obj_{\delta}(w_k)\|].
\end{align}

We now provide a bound for the last term of \eqref{eq:ncvxConsBoundMid}. Based on the arithmetic mean-quadratic mean (AM-QM) inequality, we have
\begin{equation}\label{eq:AM-QM-error}
    \begin{split}
        \frac{1}{T}\sum_{k=1}^{T}&\E_{v_{[T]}}[\|\tilde{\phi}_k-\nabla\obj_{\delta}(w_k)\|] \\
            &\leq \sqrt{\frac{1}{T}\sum_{k=1}^{T}\E_{v_{[T]}}[\|\tilde{\phi}_k-\nabla\obj_{\delta}(w_k)\|^2]}.
    \end{split}
\end{equation}
Next, we focus on the bound for the second moment $\E[\|\tilde{\phi}_k-\nabla\obj_{\delta}(w_k)\|^2]$. Note that
\begin{align}\label{eq:secondMomentIneqConsBasic}
    &\E_{v_{[k]}}[\|\tilde{\phi}_k-\nabla\obj_{\delta}(w_k)\|^2] \notag \\
    &= \E_{v_{[k]}}\left[\big\|\frac{p}{\delta}v_k(\Phi(u_k,y_{k+1})-\Phi(u_{k-1},y_k))-\nabla\obj_{\delta}(w_k)\big\|^2 \right] \notag \\
    &= \E_{v_{[k]}}\left[\big\|\frac{p}{\delta}v_k(\obj(u_k)-\obj(u_{k-1}))-\nabla\obj_{\delta}(w_k) \right. \notag \\
                    &\hspace{4.5em} \left. + \frac{p}{\delta}v_k(e_{\Phi}(x_k,u_k)-e_{\Phi}(x_{k-1},u_{k-1}))\big\|^2 \right] \notag \\
    &\stackrel{\text{(s.1)}}{\leq} \underbrace{3\E_{v_{[k]}}\left[\big\|\frac{p}{\delta}v_k(\obj(u_k)-\obj(u_{k-1}))-\nabla\obj_{\delta}(w_k)\big\|^2\right]}_{\numcircled{1}} \notag \\
            & + \underbrace{\frac{3p^2}{\delta^2}\big(\E_{v_{[k]}}[\|v_k e_{\Phi}(x_k,u_k)\|^2]+\E_{v_{[k]}}[\|v_k e_{\Phi}(x_{k-1},u_{k-1})\|^2]\big)}_{\numcircled{2}},
\end{align}
where (s.1) follows from the AM-QM inequality. Then, we provide bounds on the two terms of \eqref{eq:secondMomentIneqConsBasic} in order.

For term \numcircled{1} in \eqref{eq:secondMomentIneqConsBasic},
\begin{align*}
    \numcircled{1} &\stackrel{\text{(s.1)}}{\leq} 3\E_{v_{[k]}}\left[\big\|\frac{p}{\delta}v_k(\obj(u_k) - \obj(u_{k-1})\big\|^2\right] \notag \\
        &= \frac{3p^2}{\delta^2}\E_{v_{[k]}}\big[\|v_k(\obj(w_k+\delta v_k) - \obj(w_{k-1}+\delta v_k) \notag \\
            & \qquad\quad + \obj(w_{k-1}+\delta v_k) - \obj(w_{k-1}+\delta v_{k-1}))\|^2 \big] \notag \\
            &\stackrel{\text{(s.2)}}{\leq} \frac{6p^2}{\delta^2}\E_{v_{[k]}}\big[\|v_k(\obj(w_k+\delta v_k) - \obj(w_{k-1}+\delta v_k))\|^2 \notag \\
            & \qquad\quad + \|v_k(\obj(w_{k-1}+\delta v_k) - \obj(w_{k-1}+\delta v_{k-1}))\|^2 \big] \notag \\
            &\stackrel{\text{(s.3)}}{\leq} \frac{6M^2p^2}{\delta^2}\E_{v_{[k]}}\big[\|w_k-w_{k-1}\|^2 \|v_k\|^2 \notag \\
            &\qquad\qquad\qquad~ + \delta^2\|v_k-v_{k-1}\|^2\|v_k\|^2 \big] \notag \\
            &\stackrel{\text{(s.4)}}{\leq} \frac{6M^2p^2}{\delta^2}\big(\E_{v_{[k]}}[\eta^2\|s_{k-1}-w_{k-1}\|^2] \notag \\
            &\qquad\qquad\qquad~ + \delta^2\E_{v_{[k]}}[2(\|v_k\|^2+\|v_{k-1}\|^2)\|v_k\|^2]\big) \notag \\
            &\stackrel{\text{(s.5)}}{\leq} \frac{6M^2p^2}{\delta^2}(\eta^2D^2 + 4\delta^2).
\end{align*}
In (s.1), we utilize the property that the variance is not more than the second moment, given that
\begin{align*}
    \E_{v_{[k]}}&\left[\frac{p}{\delta}v_k(\obj(u_k) - \obj(u_{k-1})\right] \\
        & = \E_{v_{[k]}}\left[\frac{p}{\delta}v_k(\obj(w_k+\delta v_k) - \obj(w_{k-1}+\delta v_{k-1})) \right] \\
        & \stackrel{\text{(e.1)}}{=} \E_{v_{[k]}}\left[\frac{p}{\delta}v_k \obj(w_k+\delta v_k) \right] \\
        & \stackrel{\text{(e.2)}}{=} \nabla \obj_{\delta}(w_k),
\end{align*}
where (e.1) follows from the independence of $\obj(w_{k-1}+\delta v_{k-1})$ on $v_k$ and $\E_{v_{[k]}} [v_k] = 0$, and (e.2) uses \eqref{eq:expectedGradSmoothApprox}. Furthermore, (s.2) relies on the inequality $\|a+b\|^2\leq 2(\|a\|^2+\|b\|^2), \forall a,b$; (s.3) uses the assumption that $\obj(u)$ is $M$-Lipschitz; (s.4) follows from the update \eqref{eq:updateDTCons} and the inequality $\|a-b\|^{2}\leq 2(\|a\|^2+\|b\|^2), \forall a,b$; (s.5) holds because $s_{k-1},w_{k-1}\in \mathcal{U}$ and $\mathcal{U}$ is bounded with diameter $D$.

Then, we bound term \numcircled{2} in \eqref{eq:secondMomentIneqConsBasic}.
\begin{align*}
    \numcircled{2} &\stackrel{\text{(s.1)}}{=} \frac{3p^2}{\delta^2} \big(\E_{v_{[k]}}[|e_{\Phi}(x_k,u_k)|^2]+\E_{v_{[k]}}[|e_{\Phi}(x_{k-1},u_{k-1})|^2]\big) \\
            &\stackrel{\text{(s.2)}}{\leq} \frac{3p^2 M_\Phi^2 M_g^2}{\alpha_1\delta^2} \Big(1-\frac{\alpha_3}{\alpha_2}\Big) \\
            &\qquad \cdot \big(\E_{v_{[k]}}[V(x_k,u_k,d)]+\E_{v_{[k]}}[V(x_{k-1},u_{k-1},d)]\big),
\end{align*}
where (s.1) uses the fact that $\|v_k\|=1$, since $v_k\sim U(\mathbb{S}_{p-1})$; (s.2) follows from the bound on $|e_{\Phi}(x_k,u_k)|^2$ in \eqref{eq:objSquareErrorSubstitution}.


Next, we focus on the recursive inequality of the expected value of the Lyapunov function, i.e., $\E_{v_{[k]}}[V(x_k,u_k,d)]$. The inequality \eqref{eq:VdecreasePerStepExpectMid} in Appendix~\ref{subsec:expectedLyapunovFuncDecrease} still holds. Note that we use the new update rule \eqref{eq:updateOPTCons} and draw $v_k,v_{k-1}$ from the unit sphere. Hence, the bound on $\E_{v_{[k]}}[\|u_k-u_{k-1}\|^2]$ is given by
\begin{align*}
    \E_{v_{[k]}}&[\|u_k-u_{k-1}\|^2] = \E_{v_{[k]}}[\|w_k-w_{k-1}+\delta v_k-\delta v_{k-1}\|^2] \\
            &\stackrel{\text{(s.1)}}{\leq} \E_{v_{[k]}}[2\|w_k-w_{k-1}\|^2+2\delta^2\|v_k-v_{k-1}\|^2] \\
            &\stackrel{\text{(s.2)}}{\leq} \E_{v_{[k]}}[2\eta^2\|s_{k-1}-w_{k-1}\|^2+4\delta^2(\|v_k\|^2+\|v_{k-1}\|^2)] \\
            &\stackrel{\text{(s.3)}}{\leq} 2\eta^2 D^2 + 8\delta^2,
\end{align*}
where (s.1) uses the inequality $\E[(a+b)^2]\leq 2\E[a^2+b^2], \forall a,b$; (s.2) follows from the update \eqref{eq:updateDTCons} and the inequality $\|a-b\|^{2}\leq 2(\|a\|^2+\|b\|^2), \forall a,b$; (s.3) relies on the fact that $s_{k-1}$ and $w_{k-1}$ lie in the bounded set $\mathcal{U}$ with diameter $D$, and that $\|v_k\|=1$. Therefore,
\begin{align}\label{eq:expectedVdecreaseCons} 
    \E_{v_{[k]}}[V(x_k,u_k,d)] &\leq \mu \E_{v_{[k]}}[V(x_{k-1},u_{k-1},d)] \notag \\
            &\qquad + 2\alpha_2 M_x^2 (2\eta^2 D^2 + 8\delta^2).
\end{align}
It follows that
\begin{align*}
    \frac{1}{T}&\sum_{k=0}^{T-1} \E_{v_{[T]}}[V(x_k,u_k,d)] \\
        &\leq \frac{E_{v_{[T]}}[V(x_0,u_0,d)]+2(T-1)\alpha_2 M_x^2(2\eta^2D^2+8\delta^2)}{(1-\mu)T}.
\end{align*}
By incorporating the obtained bounds on terms \numcircled{1} and \numcircled{2} as well as \eqref{eq:expectedVdecreaseCons} into \eqref{eq:secondMomentIneqConsBasic}, we have
\begin{align}\label{eq:avgVarGradEst}
    \frac{1}{T} &\sum_{k=1}^{T} \E_{v_{[T]}}[\|\tilde{\phi}_k - \nabla \obj_{\delta}(w_k)\|^2] \leq \notag \\ 
        &\frac{3p^2}{\delta^2}\bigg(2M^2(\eta^2D^2+4\delta^2) + \frac{M_\Phi^2M_g^2}{\alpha_1}\big(1-\frac{\alpha_3}{\alpha_2}\big) \notag \\
        &\qquad\cdot \Big((1+\mu)\frac{1}{T}\sum_{k=0}^{T-1} \E_{v_{[T]}}[V(x_k,u_k,d)] \notag \\
        &\qquad\quad + 2\alpha_2M_x^2(2\eta^2D^2+8\delta^2)\Big)\bigg).
\end{align}
To guarantee that $|\obj_{\delta}(u) - \obj(u)|\leq \epsilon$, we set $\delta = \frac{\epsilon}{M}$. We combine \eqref{eq:ncvxConsBoundMid}, \eqref{eq:AM-QM-error} and \eqref{eq:avgVarGradEst} and obtain
\begin{align}\label{eq:boundOrderComplexityConsMid}
    \frac{1}{T}\sum_{k=1}^{T} &\E_{v_{[T]}}[\mathcal{G}(w_k)] \leq \frac{\E_{v_{[T]}}[\obj_{\delta}(w_1)]-\obj_{\delta}^*}{\eta T} + \frac{D^2Mp}{2}\cdot\frac{\eta}{\delta} \notag \\
        &+ D\bigg\{6M^2D^2p^2\frac{\eta^2}{\delta^2} + 24M^2p^2 + \frac{M_\Phi^2M_g^2}{\alpha_1}\big(1-\frac{\alpha_3}{\alpha_2}\big) \notag \\
        &\hspace{3em} \cdot \bigg[\frac{1+\mu}{(1-\mu)T}\cdot \Big(E_{v_{[T]}}[V(x_0,u_0,d)] \notag \\
                &\hspace{4em} + 2(T-1)\alpha_2 M_x^2\big(2D^2\eta^2+8\delta^2\big)\Big) \notag \\
                &\hspace{4em} + 2\alpha_2M_x^2\big(2D^2\eta^2+8\delta^2\big) \bigg] \bigg\}^\frac{1}{2},
\end{align}
where we use the fact that $\E[\obj_{\delta}(w_{T+1})]\geq \obj_{\delta}^{*}\triangleq \inf_{u\in \mathcal{U}} \obj_{\delta}(u)$. Note that $\obj_{\delta}^{*}$ is finite, since $\inf_{u\in \mathcal{U}} \obj(u)$ is finite (see \eqref{eq:infiObjCons}) and $\obj_{\delta}(u)$ is close to $\obj(u)$ (see Lemma~\ref{lem:uniformSmoothApprox}). Then, we set $\eta= \kappa\sqrt{\frac{\epsilon}{\smash{p T}}}$ such that $\frac{1}{\eta T}$ and $\frac{p\eta}{\delta}$ are of the same order. Also, $\kappa$ falls in $\big(0,\frac{\sqrt{pT}}{\sqrt{\epsilon}})$ to ensure that $\eta\in (0,1)$. Consequently, for any $k\in \mathbb{N}$, $w_{k+1}$ is a convex combination of $w_k$ and $s_k$ and lies in $\mathcal{U}$ provided that $w_0\in \mathcal{U}$. By plugging the expressions of $\delta$ and $\eta$ into \eqref{eq:boundOrderComplexityConsMid}, we arrive at \eqref{eq:boundOrderComplexityCons}.

\section*{Acknowledgment}
The authors acknowledge Dr.~Giuseppe Belgioioso, Dr.~Dominic Liao-McPherson, Dr.~Lukas Ortmann, Dr.~Keith Moffat, Dr.~Miguel Picallo, and Jialun Li for enlightening discussions.
    
    \balance
    \bibliographystyle{IEEEtran}
    \bibliography{article}

\end{document}